\documentclass[a4paper,10pt]{article}
\usepackage[utf8]{inputenc}

\usepackage{cmap}
\usepackage[T1]{fontenc}

\usepackage{subcaption}

\usepackage{tocloft}
\usepackage[english]{babel}
\usepackage{amsfonts}
\usepackage{amssymb, amsthm}
\usepackage{xcolor,graphicx}
\usepackage{marginnote}
\usepackage{amsmath}
\usepackage{a4wide}
\usepackage[affil-it]{authblk}

\usepackage{epstopdf}
\usepackage{titlesec}
\titlelabel{\thetitle.\quad}

\usepackage[labelsep=period]{caption}

\usepackage{enumitem}

\usepackage{crossreftools}
\makeatletter
\newcommand{\optionaldesc}[2]{%
	\phantomsection
	#1\protected@edef\@currentlabel{#1}\label{#2}%
}
\makeatother

\numberwithin{equation}{section}

\let\OLDthebibliography\thebibliography
\renewcommand\thebibliography[1]{
	\OLDthebibliography{#1}
	\setlength{\parskip}{1pt}
	\setlength{\itemsep}{1pt plus 0.3ex}
}

\usepackage[colorlinks=true,linktocpage,pdfpagelabels,
bookmarksnumbered,bookmarksopen]{hyperref}
\definecolor{ForestGreen}{rgb}{0.1,0.6,0.05}
\definecolor{EgyptBlue}{rgb}{0.063,0.1,0.6}
\hypersetup{
	colorlinks=true,
	linkcolor=EgyptBlue,         
	citecolor=ForestGreen,
	urlcolor=olive
}

\usepackage[hyperpageref]{backref}

\allowdisplaybreaks
\usepackage{mathtools}
\mathtoolsset{showonlyrefs}

\usepackage{orcidlink}

\usepackage{accents}

\def\Wo{\widetilde{W}_0^{s,p}(\Omega)}
\def\W{W^{s,p}(\mathbb{R}^N)}

\newtheorem{theorem}{Theorem}[section]
\newtheorem{lemma}[theorem]{Lemma}
\newtheorem{proposition}[theorem]{Proposition}
\newtheorem{corollary}[theorem]{Corollary}
\theoremstyle{definition}
\newtheorem{remark}[theorem]{Remark}

\title{Payne nodal set conjecture for the fractional $p$-Laplacian in Steiner symmetric domains
	\\[1.5em]
	
	\small{\textit{Dedicated to 90th birthday of Nina Nikolaevna Uraltseva, with deep respect and admiration}}\\[1em]	
}

\author{Vladimir Bobkov ~\&~ Sergey Kolonitskii \\}
\date{}

\AtEndDocument{%
	\par
	\medskip
	\bigskip
	\begin{tabular}{@{}l@{}}%
		(V.~Bobkov)\\[0.2em]
		\textsc{Institute of Mathematics, Ufa Federal Research Centre, RAS}\\
		\textsc{Chernyshevsky str. 112, 450008 Ufa, Russia}\\[0.3em]
		\orcidlink{0000-0002-4425-0218} 0000-0002-4425-0218\\[0.3em]
		\textit{E-mail address}: \texttt{bobkov@matem.anrb.ru, bobkovve@gmail.com}\\[1.5em]
		
		(S.~Kolonitskii)\\[0.2em]
		\textsc{Saint Petersburg Electrotechnical University ``LETI''}\\
		\textsc{Professora Popova str. 5, 197376 St. Petersburg, Russia}\\[0.3em]
		\orcidlink{0000-0002-0535-0582} 0000-0002-0535-0582\\[0.3em]
		\textit{E-mail address}: \texttt{sbkolonitskii@etu.ru, sergey.kolonitskii@gmail.com} 
\end{tabular}}

\begin{document}
	\maketitle
	\vspace*{-5ex}
	\begin{abstract}
		Let $u$ be either a second eigenfunction of the fractional $p$-Laplacian or a least energy nodal solution of the equation $(-\Delta)^s_p \, u = f(u)$ with superhomogeneous and subcritical nonlinearity $f$, in a bounded open set $\Omega$ and under the nonlocal zero Dirichlet conditions. 
		Assuming only that $\Omega$ is Steiner symmetric, we show that the supports of positive and negative parts of $u$ touch $\partial\Omega$.
		As a consequence, the nodal set of $u$ has the same property whenever $\Omega$ is connected.
		The proof is based on the analysis of equality cases in certain polarization inequalities involving positive and negative parts of $u$, and on alternative characterizations of second eigenfunctions and least energy nodal solutions.
		
		\par
		\smallskip
		\noindent {\bf  Keywords}: 
		fractional $p$-Laplacian; second eigenfunctions; least energy nodal solutions; Payne conjecture; nodal set; polarization.
		
		\noindent {\bf MSC2010}: 
		35J92,	%Quasilinear elliptic equations with $p$-Laplacian
		35R11,	%Fractional partial differential equations 
		35B06,	%Symmetries, invariants, etc. 
		49K30.	%Optimal solutions belonging to restricted classes
	\end{abstract}
	
	%ToC
	\renewcommand{\cftdot}{.}
	\begin{quote}	
		\tableofcontents	
		\addtocontents{toc}{\vspace*{-2ex}}
	\end{quote}

	\section{Introduction}\label{sec:intro}
	Let $s \in (0,1)$, $p \in (1,+\infty)$, and $\Omega$ be a bounded open set in $\mathbb{R}^N$, $N \geqslant 1$. 
	Consider the problem
	\begin{equation}\label{eq:D}
		\tag{$\mathcal{D}$}
		\left\{
		\begin{aligned}
			(-\Delta)^s_p \, u &= f(u) 
			&&\text{in } \Omega, \\
			u &= 0  &&\text{in } \mathbb{R}^N \setminus \Omega,
		\end{aligned}
		\right.
	\end{equation}
	where $(-\Delta)^s_p$ is the fractional $p$-Laplacian which can be defined for sufficiently regular functions as
	$$
	(-\Delta)^s_p u(x) 
	= 
	2
	\lim_{\varepsilon \to 0+} \int_{\mathbb{R}^N \setminus B(x,\varepsilon)} \frac{|u(x)-u(y)|^{p-2}(u(x)-u(y))}{|x-y|^{N+ps}} \, dy,
	$$
	and the nonlinearity $f: \mathbb{R} \to \mathbb{R}$ will be discussed below.
	We understand the problem \eqref{eq:D} in the following weak sense. 
	Consider the fractional Sobolev space
	$$
	\W 
	= 
	\{
	u \in L^p(\mathbb{R}^N):~ [u]_p < +\infty
	\},
	$$
	where $[\,\cdot\,]_p$ stands for the Gagliardo seminorm:
	$$
	[u]_p 
	= 
	\left(
	\int_{\mathbb{R}^N} \int_{\mathbb{R}^N} \frac{|u(x)-u(y)|^{p}}{|x-y|^{N+ps}} \, dxdy
	\right)^{1/p},
	$$
	and we will also use $\|\cdot \|_p$ for the standard norm in $L^p(\mathbb{R}^N)$.
	We denote by $\Wo$ the completion of $C_0^\infty(\Omega)$ with respect to the norm $\|\cdot\|_p+[\,\cdot\,]_p$ of $\W$. 
	This space is uniformly convex, separable, Banach space with the norm $[\,\cdot\,]_p$, and the embedding $\Wo \hookrightarrow L^p(\Omega)$ is compact, see \cite{BLP}.
	Weak solutions of \eqref{eq:D} are critical points of the energy functional $E: \Wo \to \mathbb{R}$ defined as
	$$
	E(u) = \frac{1}{p} \, [u]_p^p - \int_\Omega F(u) \, dx,
	$$
	where $F(z)=\int_0^z f(t)\,dt$.
	In other words, a function $u \in \Wo$ is a (weak) solution of \eqref{eq:D} if
	\begin{equation}\label{eq:DEu}
		\langle D E(u),\xi \rangle
		\equiv 
		\frac{1}{p} \,\langle D[u]_p^p, \xi \rangle 
		-
		\int_\Omega f(u) \xi \, dx
		= 0
		\quad \text{for any}~ \xi \in \Wo.
	\end{equation}
	For convenience, we note explicitly that
	\begin{equation}\label{eq:Dupxi}
		\langle D[u]_p^p, \xi \rangle 
		=
		p \,
		\int_{\mathbb{R}^N} \int_{\mathbb{R}^N} \frac{|u(x)-u(y)|^{p-2}(u(x)-u(y))(\xi(x)-\xi(y))}{|x-y|^{N+ps}} \, dxdy.
	\end{equation}
	
	\medskip
	We assume that the nonlinearity $f$ is either of the following two types:
	\begin{enumerate}
		\item[{\optionaldesc{$(\mathcal{F}_r)$}{Fr}}]
		The \textit{resonant} case $f(z) = \lambda_2 |z|^{p-2}z$, where $\lambda_2$ is the second eigenvalue of the fractional Dirichlet $p$-Laplacian in $\Omega$.
		This eigenvalue can be characterized as 
		\begin{equation}\label{eq:lambda2}
			\lambda_2 = \inf_{\mathcal{A} \subset \mathcal{G}_2} \sup_{u \in \mathcal{A}} [u]_p^p,
		\end{equation}
		see, e.g., \cite{BrPar,drabekrobinson}, 
		where we denote
		\begin{align}
			\mathcal{G}_2 &= \{\mathcal{A} \subset \mathcal{S}:~\text{there exists a continuous odd surjection } h: S^1 \to \mathcal{A}\},\\
			\label{eq:Slp}
			\mathcal{S} &= \{u \in \Wo:~ \|u\|_p = 1\},
		\end{align}
		and $S^1$ stands for a circle in $\mathbb{R}^2$.
		
		\item[{\optionaldesc{$(\mathcal{F}_s)$}{Fs}}]
		The \textit{superhomogeneous} and subcritical 
		case characterized by the following assumptions:
		\begin{enumerate}[label={\rm(\alph*)}]
			\item\label{Fs-a} $f \in C(\mathbb{R}) \cap C^1(\mathbb{R} \setminus \{0\})$ and there exist constants $C>0$ and $q \in (p,p_s^*)$, where $p_s^* = \frac{Np}{N-ps}$ when $N > ps$ and $p_s^* = +\infty$ when $N \leqslant ps$, such that
			$$
			|f(z)| \leqslant C \, (1+|z|^{q-1})
			\quad \text{for any}~ z \in \mathbb{R}.
			$$
			\item\label{Fs-b} 
			The function $z \mapsto \frac{f(z)}{|z|^{p-2}z}$ is decreasing in $(-\infty,0)$ and increasing in $(0,+\infty)$ (both monotonicities are strict), and
			$$
			\lim_{|z| \to 0} \frac{f(z)}{|z|^{p-2}z} = 0
			\quad \text{and} \quad 
			\lim_{|z| \to +\infty} \frac{f(z)}{|z|^{p-2}z} = +\infty.
			$$
		\end{enumerate}
		The model case of such nonlinearity is $f(z) = |z|^{q-2}z$ with $q \in (p, p^*_s)$, and the inequality $q>p$ motivates the word ``superhomogeneous'' (''superlinear'' in the case $p=2$).
		The assumptions \ref{Fs-a} and \ref{Fs-b} are not minimal and can be weakened to some extent, see Remark~\ref{rem:Fs}.
	\end{enumerate}
	
	Under the resonance assumption~\ref{Fr}, the problem \eqref{eq:D} has a solution which is naturally called \textit{second eigenfunction}, see, e.g., \cite[Section~3]{Servadei2} for the case $p=2$ and
	\cite{BrPar} for the case $p>1$, and we refer to \cite{BLP,FranPal,HS,lindgren} for further results on the spectrum of the fractional $p$-Laplacian.
	Let us explicitly note that any second eigenfunction changes sign in $\Omega$, see \cite{BrPar}.
	
	Under the superhomogeneous assumption~\ref{Fs}, the problem \eqref{eq:D} admits nodal solutions (i.e., sign-changing solutions), among which we will be interested only in solutions having minimal value of the functional $E$ among all other nodal solutions.
	Solutions with this property are called \textit{least energy nodal solutions} (LENS, for brevity). 
	The existence of LENS was established in the linear case $p=2$ in 
	\cite{GST} for the model nonlinearity $f(z) = |z|^{q-2}z$ and in
	\cite{GTZ,GYZ,luo,tengwangwang} for more general nonlinearities satisfying assumptions similar to \ref{Fs}. 
	In the nonlinear case $p>1$, corresponding existence results were obtained in 
	\cite{CNW,WZ}, and \ref{Fs} comes from \cite{CNW}\footnote{Note that \cite[Theorem~1.2]{CNW} additionally requires $\Omega$ to be a smooth domain, $1<p<N/s$, $N \geqslant 2$, and $f \in C^1(\mathbb{R})$.
		However, thanks to the properties of the space $\Wo$, inspection of proofs from \cite{CNW} shows that all the results from \cite[Section~4]{CNW} and hence \cite[Theorem~1.2]{CNW} remain valid under the present weaker assumptions.}.
	Least energy nodal solutions and tightly connected with second eigenfunctions and can be seen as objects of the same nature, see \cite{GST} for some asymptotic results.
	An important property of LENS in the superhomogeneous case~\ref{Fs}, which will be heavily employed in our arguments, is the fact that such solutions can be characterized as minimizers of $E$ over the so-called nodal Nehari set, see Section~\ref{section:aux2}.
	Analogous constructive variational characterization does not hold in the subhomogeneous regime (cf.\ \cite{BMPTW}), that is why we do not cover it.
	
	Throughout the text, we decompose a function $w \in \W$ as
	\begin{equation}\label{eq:notationpm}
		w = w^+ + w^-,
		\quad \text{where} \quad
		w^+ = \max\{u,0\}
		\quad \text{and} \quad 
		w^- = \min\{u,0\}.
	\end{equation} 
	In particular, $w^\pm \in \W$ and $w^+ \geqslant 0$, $w	^- \leqslant 0$ a.e.\ in $\mathbb{R}^N$.
	
	It is known that any second eigenfunction or LENS $u$ is continuous in $\Omega$, see, e.g., the combination of an $L^\infty$-bound \cite[Theorem~4.1]{HS} with a local H\"older estimate \cite[Corollary~5.5]{IMS}, see also \cite{BrPar,KMS}.
	With this regularity at hand, we define the \textit{nodal set} of $u$ as
	\begin{equation*}
		\mathcal{Z}(u) = \overline{\{x \in \Omega:~ u(x) = 0\}}.
	\end{equation*}
	Note that $\mathcal{Z}(u)$ might be empty if $\Omega$ is not connected. 
	For instance, this is a likely behavior when $\Omega$ is a disjoint union of two equimeasurable balls, cf.~\cite{BrPar}.
	At the same time, if $\Omega$ is connected (i.e., $\Omega$ is a domain), then $\mathcal{Z}(u) \neq \emptyset$.

	We are interested in properties of the nodal set and supports of positive and negative parts of second eigenfunctions and LENS, and establish the following result.
	\begin{theorem}\label{thm:payne}
		Assume that $\Omega$ is Steiner symmetric with respect to the hyperplane 
		$$
		H_0 = \{(x_1,\dots,x_N) \in \mathbb{R}^N:~x_1 = 0\}.
		$$
		Let $u$ be either a second eigenfunction or a least energy nodal solution of \eqref{eq:D}. 
		Then
		\begin{equation}\label{eq:dist0x}
			\mathrm{dist}(\mathrm{supp}\, u^-, \partial \Omega) = 0
			\quad \text{and} \quad
			\mathrm{dist}(\mathrm{supp}\, u^+, \partial \Omega) = 0.
		\end{equation}
		Consequently, if $\Omega$ is connected, then
		\begin{equation}\label{eq:dist0}
			\mathrm{dist}(\mathcal{Z}(u), \partial \Omega) = 0.
		\end{equation}
	\end{theorem}
	
	\begin{remark}
		The assumption that $\Omega$ is Steiner symmetric with respect to the hyperplane $H_0$ is equivalent to saying that $\Omega$ is convex with respect to the $x_1$-axis and symmetric with respect to $H_0$. 
		The convexity with respect to the $x_1$-axis means that any segment parallel to the first coordinate vector $e_1$ with endpoints in $\Omega$ is fully contained in $\Omega$, see, e.g., Figure~\ref{fig:proof}.
	\end{remark}
	
	In the \textit{local} \textit{linear} case $s=1$, $p=2$, 
	in which $(-\Delta)^s_p$ corresponds to the standard Laplace operator, 
	the assertion \eqref{eq:dist0} for any second eigenfunction $u$ in a domain $\Omega$ is the content of the famous Payne nodal set conjecture \cite{payne1973}.
	In fact, thanks to the Pohozaev identity, \eqref{eq:dist0} is equivalent to \eqref{eq:dist0x}.
	It is known that, in general, Payne's conjecture is not true, since there exist domains whose second eigenfunction $u$ satisfies $\mathrm{dist}(\mathcal{Z}(u), \partial \Omega) > 0$, see \cite{fournais,HoffmannOstenhof} and references to these works.
	Nevertheless, the conjecture is valid on certain classes of domains.
	This was established in \cite{damascelliNodal,payne1973} for Steiner symmetric domains as in Theorem~\ref{thm:payne} (by different methods), and we also refer to \cite{aftalion,FK,GrumiauTroestler2009,kiwan,melas} for some other classes of domains, as well as for the superlinear case \ref{Fr}.
	
	In the \textit{local} \textit{nonlinear} case $s=1$, $p>1$, the validity of Payne's conjecture for second eigenfunctions and LENS was proved in \cite{BK1} for Steiner symmetric domains under certain additional regularity assumptions on the boundary. 
	We are not aware of other results on Payne's conjecture in local nonlinear settings, but we refer to \cite{ADS1,BartschWethWillem,BDG,BK2,ChG} for some related results.
	
	In the \textit{nonlocal} \textit{linear} case $s \in (0,1)$, $p=2$, it was shown in \cite{BanKul,BBDG,DKK,FFTW,F} that second eigenfunctions in the ball are anti-symmetric with respect to central sections of the ball. 
	This result can be interpreted as the validity of Payne's conjecture in the ball.
	Thanks to these references, \cite{GST} guarantees that LENS share the same symmetry in the model superlinear case $f(z)=|z|^{q-2}z$ with $q \to 2$, and hence they also satisfy Payne's conjecture in the ball. 
	Again, up to our knowledge, we are not aware of such a result for other domains, even for spherical shells (cf.\ \cite{djitte-jarohs}, where the nonradiality of second eigenfunctions is shown for sufficiently thin spherical shells). 
	
	In the \textit{nonlocal} \textit{nonlinear} case $s \in (0,1)$, $p>1$, the results obtained in the present work seem to be the first on the Payne conjecture.
	It makes sense to mention, however, that Theorem~\ref{thm:payne} \textit{does not imply}, at least directly, that second eigenfunctions in the ball are \textit{nonradial}.
	Instead, Theorem~\ref{thm:payne} only yields that if a second eigenfunction is radial, then it has to oscillate near the boundary. 		
	In \cite{BBDG}, the Pohozaev identity from \cite{ROS} was used to show that such an unlikely behavior is indeed impossible, which resulted in the nonradiality of second eigenfunctions.
	Perhaps, similar strategy is applicable in the present nonlocal nonlinear settings, but the proper version of the Pohozaev identity is unknown to us. 
	Also, due to the absence of the Pohozaev identity, we cannot conclude that \eqref{eq:dist0x} is equivalent to \eqref{eq:dist0} when $\Omega$ is connected, in contrast to the local nonlinear and nonlocal linear cases. 
	That is, \eqref{eq:dist0x} is \textit{a priori} a stronger result than \eqref{eq:dist0}, at least when $s \in (0,1)$ and $p \neq 2$.

	\medskip
	
	The proof of Theorem~\ref{thm:payne} is based on two main ingredients - the polarization of functions (also known as the two-point rearrangement), and convenient characterizations of second eigenfunctions and LENS. 
	Both of these auxiliary results might be interesting in their own.
	On the fundamental level, the idea of the proof of Theorem~\ref{thm:payne} is analogous to that in \cite{BK1}, where a related result was established in the local nonlinear case.
	Later, this idea was developed in \cite{BBDG} for the nonlocal linear case when $\Omega$ is a ball, in which it has its own features and difficulties.
	In the present work, we further develop the approach from \cite{BBDG,BK1} to the nonlocal nonlinear case in Steiner symmetric sets, and note that, apart from the very general strategy, our proofs are different than those in \cite{BBDG,BK1}.
	In particular, using Proposition~\ref{prop:polarization}, the proof of \cite[Theorem~1.1]{BBDG} concerning the resonant linear case $f(z)=\lambda_2 z$ can be given in a simpler and more universal way.
	
	\medskip
	The rest of the article is organized as follows.
	Section~\ref{section:aux} is devoted to establishing certain inequalities for polarizations of functions with explicit information on equality cases, and in Section~\ref{section:aux2} we provide alternative characterizations of second eigenfunctions and LENS.
	Section~\ref{section:proof} contains the proof of our main result, Theorem~\ref{thm:payne}.
	Finally, Appendix~\ref{section:appendix} contains a few technical lemmas needed for the proofs from Sections~\ref{section:aux} and~\ref{section:aux2}.

	\section{Polarization inequalities}\label{section:aux}
	
	In this section, we deal with the classical systematization method called  \textit{polarization} (or, equivalently, \textit{two-point rearrangement}) of functions, see, e.g., \cite{baernstain,BartschWethWillem,BE,BK1,ashok1}. 
	Consider a hyperplane $H_a = \{x \in \mathbb{R}^N:\, x_1 = a \}$ where $x = (x_1,x_2,\dots,x_N)$, $a \in \mathbb{R}$, and let $\sigma_a(x) = (2a - x_1, x_2, \dots, x_N)$ be the reflection of a point $x$ with respect to $H_a$. 
	Denote the corresponding open half-spaces as 
	\begin{equation}\label{eq:Sigma}
		\Sigma_a^+ = \{ x \in \mathbb{R}^N:~ x_1 > a \} 
		\quad \text{and} \quad 
		\Sigma_a^- = \{ x \in \mathbb{R}^N:~ x_1 < a \}.
	\end{equation}
	Let $u: \mathbb{R} \to \mathbb{R}$ be a given measurable function. 
	The polarization of $u$ with respect to $H_a$ is a function $P_a u: \mathbb{R} \to \mathbb{R}$ defined as
	\begin{equation}\label{eq:Pol}
		P_a u(x) = 
		\begin{cases}
			\min \{u(x), u(\sigma_a(x))\}, &x \in \Sigma_a^+,\\
			u(x), &x \in H_a,\\
			\max \{u(x), u(\sigma_a(x))\}, &x \in \Sigma_a^-.
		\end{cases}
	\end{equation}
	It is known that $(P_a u)^\pm = P_a(u^\pm)$ in $\mathbb{R}^N$, see \cite[Lemma~2.1]{BartschWethWillem}, so hereinafter will write $P_a u^\pm$, for short.
	As a consequence, we have
	\begin{equation}\label{eq:espantion_of_polarization_positive_negative} 
		P_a (u^+ + u^-)
		= 
		P_a u
		=
		(P_a u)^+ + (P_a u)^-
		= 
		P_a u^+ + P_a u^-.
	\end{equation}
	Moreover, it is not hard to see that
	$$
	P_a u = u
	\quad \text{if and only if} \quad
	P_a u^+ = u^+ 
	~\text{and}~
	P_a u^- = u^-.
	$$
	For convenience, we define an ``opposite'' polarization $\widetilde{P}_a u$ of $u$ as
	\begin{equation}\label{eq:Pol2}
		\widetilde{P}_a u(x) = 
		\begin{cases}
			\max \{u(x), u(\sigma_a(x))\}, &x \in \Sigma_a^+,\\
			u(x), &x \in H_a,\\
			\min \{u(x), u(\sigma_a(x))\}, &x \in \Sigma_a^-.
		\end{cases}
	\end{equation}
	We observe that $P_a u^- = -\widetilde{P}_a (-u^-)$ and hence \eqref{eq:espantion_of_polarization_positive_negative} yields
	\begin{equation}\label{eq:polarization_positive_negative2} 
		P_a u
		=
		P_a u^+
		-
		\widetilde{P}_a (-u^-).
	\end{equation}
	
	The polarization can be also used to polarize sets.
	Namely, for a measurable set $\Omega$, $P_a \Omega$ and $\widetilde{P}_a \Omega$ are defined as
	\begin{equation*}
		P_a \Omega = 
		\begin{cases}
			\Omega \cap \sigma_a(\Omega) &\text{in } \Sigma_a^+,\\
			\Omega &\text{on } H_a,\\
			\Omega \cup \sigma_a(\Omega) &\text{in } \Sigma_a^-,
		\end{cases}
		\quad 
		\text{and}
		\quad
		\widetilde{P}_a \Omega = 
		\begin{cases}
			\Omega \cup \sigma_a(\Omega) &\text{in } \Sigma_a^+,\\
			\Omega &\text{on } H_a,\\
			\Omega \cap \sigma_a(\Omega) &\text{in } \Sigma_a^-.
		\end{cases}
	\end{equation*}
	Note that $\widetilde{P}_a \Omega = \mathbb{R}^N \setminus (P_a (\mathbb{R}^N \setminus \Omega))$. 		
	We refer to \cite[Proposition~2.2]{ashok1} and \cite[Section~2]{BK1} for an overview of main properties of $P_a \Omega$, see also \cite{brocksol}.
	
	The polarization has the following useful properties which will be important for us.
	First, \cite[Lemma~2.2]{BartschWethWillem} (see also \cite[Eq.~(3.7)]{brocksol}) yields
	\begin{equation}\label{eq:weak-pol}
		\int_{\mathbb{R}^N} f(P_a u^\pm) P_a u^\pm\,dx 
		= 
		\int_{\mathbb{R}^N}  f(u^\pm) u^\pm \,dx 
		\quad
		\text{and} 
		\quad
		\int_{\mathbb{R}^N} F(P_a u^\pm) \,dx = \int_{\mathbb{R}^N} F(u^\pm) \,dx.
	\end{equation}
	Second, if $u \in \W$, then $P_a u \in \W$, and we have
	\begin{equation}\label{eq:polineq1}
		[P_a u]_p
		\leqslant 
		[u]_p,
	\end{equation}
	see, e.g., \cite{baernstain,BE,brocksol,ashok1}.
	In fact, equality holds in \eqref{eq:polineq1} if and only if either $u(x)=P_a u(x)$ for a.e.\ $x \in \mathbb{R}^N$ or $u(\sigma_a(x))=P_a u(x)$ for a.e.\ $x \in \mathbb{R}^N$, see Remark~\ref{rem:equality}.
	Here, we also refer to \cite[Corollary, p.~4819]{BE}, but the reference requires $p \geqslant 2$ (and does not contain a proof),  and to \cite[Theorem~2.9]{baernstain}, but the reference requires $u$ to be nonnegative; see also \cite[Theorem~3.1]{ashok1} for a related result under different assumptions.
	
	The main aim of the present section is to provide a certain quantification of the inequality \eqref{eq:polineq1} with an explicit discussion of equality cases, thereby 
	extending and improving \cite[Lemma~2.3]{BBDG}.
	\begin{proposition}\label{prop:polarization}
		Let $a \in \mathbb{R}$ and $u \in \W$.
		Then
		\begin{align}
			\label{eq:improvement}
			\langle D[P_a u]_p^p, P_a u^+ \rangle 
			\leqslant
			\langle D[u]_p^p, u^+ \rangle,\\
			\label{eq:improvement2}
			\langle D[P_a u]_p^p, P_a u^- \rangle 
			\leqslant
			\langle D[u]_p^p, u^- \rangle.
		\end{align}
		Moreover, equality takes place in \eqref{eq:improvement} (respectively, in \eqref{eq:improvement2}) if and only if either of the following cases holds:
		\begin{enumerate}[label={\rm(\roman*)}]
			\item\label{prop:polarization:1} $u(x)=P_a u(x)$ for a.e.\ $x \in \mathbb{R}^N$;
			\item\label{prop:polarization:2} $u(\sigma_a(x))=P_a u(x)$ for a.e.\ $x \in \mathbb{R}^N$;
			\item\label{prop:polarization:3} $u^+(x) = u^+(\sigma_a(x))$ for a.e.\ $x \in \mathbb{R}^N$ (respectively, $u^-(x) = u^-(\sigma_a(x))$ for a.e.\ $x \in \mathbb{R}^N$).
		\end{enumerate}
	\end{proposition}
	\begin{proof}
		Let us start with the inequality~\eqref{eq:improvement} and assume that $u^+ \not\equiv 0$ in $\Omega$.
		Throughout the proof, we denote, for brevity,
		\begin{equation}\label{eq:J}
			v = P_a u
			\quad \text{and} \quad  
			J(\alpha,\beta) = |\alpha - \beta|^{p-2}(\alpha - \beta)(\alpha^+ - \beta^+).
		\end{equation}
		We get from \eqref{eq:espantion_of_polarization_positive_negative}  that $v^\pm = P_a u^\pm$, so that \eqref{eq:improvement} is equivalent to
		\begin{equation}\label{eq:prop:mainineq1}
			\int_{\mathbb{R}^N}\int_{\mathbb{R}^N} \frac{J(v(x),v(y))}{|x-y|^{N+sp}} \, dxdy
			\leqslant
			\int_{\mathbb{R}^N}\int_{\mathbb{R}^N} \frac{J(u(x),u(y))}{|x-y|^{N+sp}} \, dxdy,
		\end{equation}
		cf.\ \eqref{eq:Dupxi}.
		Decomposing $\mathbb{R}^N = \Sigma_a^+ \cup H_a \cup \Sigma_a^-$ and noting that $H_a$ has zero $N$-measure, we get
		\begin{align}
			\int_{\mathbb{R}^N}\int_{\mathbb{R}^N} &\frac{J(v(x),v(y))}{|x-y|^{N+sp}} \, dxdy
			\\
			&=
			\int_{\Sigma_a^+}\int_{\Sigma_a^+} \frac{J(v(x),v(y))}{|x-y|^{N+sp}} \, dxdy
			+
			\int_{\Sigma_a^+}\int_{\Sigma_a^+} \frac{J(v(\sigma_a(x)),v(y))}{|\sigma_a(x)-y|^{N+sp}} \, dxdy
			\\
			\label{eq:prop:decomp1}				
			&+
			\int_{\Sigma_a^+}\int_{\Sigma_a^+} \frac{J(v(x),v(\sigma_a(y)))}{|x-\sigma_a(y)|^{N+sp}} \, dxdy
			+
			\int_{\Sigma_a^+}\int_{\Sigma_a^+} \frac{J(v(\sigma_a(x)),v(\sigma_a(y)))}{|\sigma_a(x)-\sigma_a(y)|^{N+sp}} \, dxdy,
		\end{align}
		and an analogous representation holds for the right-hand side of \eqref{eq:prop:mainineq1}.
		Thus, in order to prove \eqref{eq:prop:mainineq1}, it is sufficient to establish the inequality
		\begin{align}
			\notag
			&\frac{J(v(x),v(y))}{|x-y|^{N+sp}} + \frac{J(v(\sigma_a(x)),v(y))}{|\sigma_a(x)-y|^{N+sp}} +  \frac{J(v(x),v(\sigma_a(y)))}{|x-\sigma_a(y)|^{N+sp}} + \frac{J(v(\sigma_a(x)),v(\sigma_a(y)))}{|\sigma_a(x)-\sigma_a(y)|^{N+sp}}
			\\
			\label{eq:8}
			&\leqslant
			\frac{J(u(x),u(y))}{|x-y|^{N+sp}} + \frac{J(u(\sigma_a(x)),u(y))}{|\sigma_a(x)-y|^{N+sp}} +  \frac{J(u(x),u(\sigma_a(y)))}{|x-\sigma_a(y)|^{N+sp}} + \frac{J(u(\sigma_a(x)),u(\sigma_a(y)))}{|\sigma_a(x)-\sigma_a(y)|^{N+sp}}
		\end{align}
		for a.e.\ $x, y \in \Sigma_a^+$, and characterize equality cases.
		Hereinafter in the proof, under $u$ and $v$ we understand some fixed representatives of corresponding equivalence classes from $\W$, so that \eqref{eq:8} makes sense for \textit{every} $x, y \in \Sigma_a^+$ at which $u$, $v$, and their reflections are defined, and we will be interested only in such $x,y$, while the $N$-measure of other points $x,y$ is zero anyway.			
		
		It is not hard to observe that
		\begin{equation}\label{eq:7}
			\frac{1}{|x-y|^{N+sp}} 
			= 
			\frac{1}{|\sigma_a(x)-\sigma_a(y)|^{N+sp}}
			> 
			\frac{1}{|\sigma_a(x)-y|^{N+sp}} 
			= 
			\frac{1}{|x-\sigma_a(y)|^{N+sp}}
		\end{equation}
		for every $x, y \in \Sigma_a^+$, since $\Sigma_a^+$ is defined by the strict inequality ``$>$'', see \eqref{eq:Sigma}. (In other words, the inequality in \eqref{eq:7} turns to equality if and only if $x \in H_a$ or $y \in H_a$.) 
		We will also need the following consequence of \eqref{eq:7}:
		\begin{equation}\label{eq:7x}
			\frac{1}{|x-y|^{N+sp}} 
			-
			\frac{1}{|x-\sigma_a(y)|^{N+sp}}
			= 
			\frac{1}{|\sigma_a(x)-\sigma_a(y)|^{N+sp}} 
			-
			\frac{1}{|\sigma_a(x)-y|^{N+sp}} > 0.
		\end{equation}
		
		Let us represent $\Sigma_a^+ \times \Sigma_a^+$ as a (nondisjoint) union of the following four subsets:
		\begin{align}
			A_{++} &= \{(x,y) \in \Sigma_a^+ \times \Sigma_a^+:~ 
			u(\sigma_a(x)) \geqslant u(x) 
			~\text{and}~
			u(\sigma_a(y)) \geqslant u(y) 
			\},\\
			A_{--} &= \{(x,y) \in \Sigma_a^+ \times \Sigma_a^+:~ 
			u(\sigma_a(x)) \leqslant u(x) 
			~\text{and}~
			u(\sigma_a(y)) \leqslant u(y) 
			\},\\
			\label{eq:Apm}
			A_{+-} &= \{(x,y) \in \Sigma_a^+ \times \Sigma_a^+:~ 
			u(\sigma_a(x)) > u(x) 
			~\text{and}~
			u(\sigma_a(y)) < u(y) 
			\},\\
			\label{eq:Amp}
			A_{-+} &= \{(x,y) \in \Sigma_a^+ \times \Sigma_a^+:~ 
			u(\sigma_a(x)) < u(x) 
			~\text{and}~
			u(\sigma_a(y)) > u(y) 
			\},
		\end{align}
		and investigate the inequality \eqref{eq:8} in each subset separately.
		
		$\bullet$
		Take any $(x,y) \in A_{++}$.
		By the definition \eqref{eq:Pol}, the polarization does not exchange values of $u$, so that
		\begin{equation}\label{eq:v=u1}
			v(x) = u(x), \quad v(\sigma_a(x)) = u(\sigma_a(x)),
			\quad \text{and} \quad 
			v(y) = u(y), \quad v(\sigma_a(y)) = u(\sigma_a(y)),
		\end{equation}
		and hence the inequality \eqref{eq:8} holds as equality for $(x,y) \in A_{++}$.
		
		$\bullet$
		Take any $(x,y) \in A_{--}$. 
		We see from \eqref{eq:Pol} that the polarization exchanges values of $u$, i.e.,
		\begin{equation}\label{eq:v=u2}
			v(x) = u(\sigma_a(x)), \quad v(\sigma_a(x)) = u(x),
			\quad \text{and} \quad 
			v(y) = u(\sigma_a(y)), \quad v(\sigma_a(y)) = u(y).
		\end{equation}
		Thus, using the equalities from \eqref{eq:7}, we rewrite the left-hand side of \eqref{eq:8} as 
		\begin{equation}\label{lem:prof:x}
			\frac{J(u(\sigma_a(x)),u(\sigma_a(y)))}{|\sigma_a(x)-\sigma_a(y)|^{N+sp}} +  \frac{J(u(x),u(\sigma_a(y)))}{|x-\sigma_a(y)|^{N+sp}} + 
			\frac{J(u(\sigma_a(x)),u(y))}{|\sigma_a(x)-y|^{N+sp}} +
			\frac{J(u(x),u(y))}{|x-y|^{N+sp}}.
		\end{equation}
		This expression coincides with the right-hand side of \eqref{eq:8}, i.e., \eqref{eq:8} holds as equality for $(x,y) \in A_{--}$.
		
		$\bullet$
		Take any $(x,y) \in A_{+-}$.
		In this case, \eqref{eq:Pol} implies that
		$$
		v(x) = u(x), \quad v(\sigma_a(x)) = u(\sigma_a(x)),
		\quad \text{and} \quad 
		v(y) = u(\sigma_a(y)), \quad v(\sigma_a(y)) = u(y),
		$$
		and we rewrite \eqref{eq:8} as
		\begin{align}
			\notag
			&\frac{J(u(x),u(\sigma_a(y)))}{|x-y|^{N+sp}} + \frac{J(u(\sigma_a(x)),u(\sigma_a(y)))}{|\sigma_a(x)-y|^{N+sp}} +  \frac{J(u(x),u(y))}{|x-\sigma_a(y)|^{N+sp}} + \frac{J(u(\sigma_a(x)),u(y))}{|\sigma_a(x)-\sigma_a(y)|^{N+sp}}\\
			&\leqslant
			\frac{J(u(x),u(y))}{|x-y|^{N+sp}} + \frac{J(u(\sigma_a(x)),u(y))}{|\sigma_a(x)-y|^{N+sp}} +  \frac{J(u(x),u(\sigma_a(y)))}{|x-\sigma_a(y)|^{N+sp}} + \frac{J(u(\sigma_a(x)),u(\sigma_a(y)))}{|\sigma_a(x)-\sigma_a(y)|^{N+sp}}.
			\label{eq:FFFF}
		\end{align}
		By rearranging the terms in \eqref{eq:FFFF}, we get
		\begin{align}
			&J(u(x),u(y))
			\left(
			\frac{1}{|x-y|^{N+sp}}
			-
			\frac{1}{|{x}-\sigma_a(y)|^{N+sp}}
			\right)
			\\
			&-
			J(u(\sigma_a(x)),u(y))
			\left(
			\frac{1}{|\sigma_a(x)-\sigma_a(y)|^{N+sp}}
			-
			\frac{1}{|\sigma_a(x)-y|^{N+sp}}
			\right)
			\\
			&-
			J(u(x),u(\sigma_a(y)))
			\left(
			\frac{1}{|x-y|^{N+sp}}
			-
			\frac{1}{|{x}-\sigma_a(y)|^{N+sp}}
			\right)
			\\
			&+
			J(u(\sigma_a(x)),u(\sigma_a(y)))
			\left(
			\frac{1}{|\sigma_a(x)-\sigma_a(y)|^{N+sp}}
			-
			\frac{1}{|\sigma_a(x)-y|^{N+sp}}
			\right) \geqslant 0.
			\label{eq:FFFF2}
		\end{align}
		Applying the equality in \eqref{eq:7x}, we rewrite \eqref{eq:FFFF2} as
		\begin{align}
			\left[
			J(u(x),u(y)) - J(u(\sigma_a(x)),u(y)) - J(u(x),u(\sigma_a(y))) + J(u(\sigma_a(x)),u(\sigma_a(y)))
			\right] \\
			\label{four_point_formula0}
			\times
			\left(
			\frac{1}{|x-y|^{N+sp}}
			-
			\frac{1}{|{x}-\sigma_a(y)|^{N+sp}}
			\right)
			\geqslant 0.
		\end{align}
		Thanks to the inequality in \eqref{eq:7x}, we conclude that \eqref{four_point_formula0} (and hence \eqref{eq:8}) is equivalent to the following \textit{four-point inequality}:
		\begin{equation}\label{four_point_formula}
			J(u(x),u(y)) - J(u(\sigma_a(x)),u(y)) - J(u(x),u(\sigma_a(y))) + J(u(\sigma_a(x)),u(\sigma_a(y)))
			\geqslant 0.
		\end{equation}
		This inequality is proved in Lemma~\ref{lem:four_point_lemma} by taking $a = u(x)$, $A = u(\sigma_a(x))$, $b = u(\sigma_a(y))$, $B = u(y)$.
		Moreover, Lemma~\ref{lem:four_point_lemma} implies that \eqref{four_point_formula} is strict if and only if $u(\sigma_a(x)) > 0$ or $u(y) > 0$.
		Consequently, if $(x,y) \in A_{+-}$ is such that $u(\sigma_a(x)) > 0$ or $u(y) > 0$, then the inequality \eqref{eq:8} holds with the strict sign. 
		We denote the set of such points as $A_{+-}^*$, i.e., 
		$$
		A_{+-}^* = \{(x,y) \in A_{+-}:~ u(\sigma_a(x)) > 0 ~~\text{or}~~ u(y) > 0\}.
		$$			
		For all $(x,y) \in A_{+-} \setminus A_{+-}^*$, \eqref{eq:8} holds with the equality sign.

		$\bullet$ 
		Take any $(x,y) \in A_{-+}$.
		Switching the notation $x \leftrightarrow y$, we arrive at the previous case, and hence deduce that if $(x,y) \in A_{-+}$ is such that $u(x) > 0$ or $u(\sigma_a(y)) > 0$, then the inequality \eqref{eq:8} holds with the strict sign, while for all other $(x,y) \in A_{-+}$, \eqref{eq:8} holds with the equality sign. 
		We denote 
		$$
		A_{-+}^* = \{(x,y) \in A_{-+}:~ u(x) > 0 ~~\text{or}~~ u(\sigma_a(y)) > 0\}.
		$$
		
		Combining all four cases, we conclude that \eqref{eq:8} is satisfied for all $x, y \in \Sigma_a^+$, which proves \eqref{eq:improvement}.
		It remains to describe the occurrence of equality in \eqref{eq:improvement}.
		We distinguish two cases:
		
		1) Let $|A_{+-}|_{2N}=0$, where $|\cdot|_{2N}$ stands for the $2N$-measure. (Equivalently, one can assume $|A_{-+}|_{2N}=0$, since the sets $A_{+-}$ and $A_{-+}$ are symmetric.)
		Consequently, we have either $u(\sigma_a(x)) \geqslant u(x)$
		for a.e.\ $x \in \Sigma_a^+$, or $u(\sigma_a(x)) \leqslant u(x)$
		for a.e.\ $x \in \Sigma_a^+$.
		This is the same as the alternative: either \eqref{eq:v=u1} holds for a.e.\ $x,y \in \Sigma_a^+$, or \eqref{eq:v=u2} holds for a.e.\ $x,y \in \Sigma_a^+$.
		In either case, we have equality in \eqref{eq:8} for a.e.\ $x,y \in \Sigma_a^+$, which results in the equality in \eqref{eq:improvement}, and \ref{prop:polarization:1} or \ref{prop:polarization:2} holds.

		2) Let $|A_{+-}|_{2N}>0$.
		For convenience, denote the left- and right-hand sides of \eqref{eq:8} as $I(v)$ and $I(u)$, respectively.
		With these notation, the inequality \eqref{eq:improvement} (via 			\eqref{eq:prop:mainineq1} and \eqref{eq:prop:decomp1}) reads as
		$$
		\int_{\Sigma_a^+}\int_{\Sigma_a^+}  (I(v) - I(u)) \,dxdy
		\leqslant
		0.
		$$
		Using the properties of the sets $A_{++}$, $A_{--}$, $A_{+-}$, $A_{-+}$, provided above, we have
		\begin{align}
			\int_{\Sigma_a^+}\int_{\Sigma_a^+}  (I(v) - I(u)) \,dxdy
			\label{eq:prop:lastinequ1}
			&=
			\iint_{A_{+-}^*} (I(v) - I(u)) \,dxdy
			+
			\iint_{A_{-+}^*}  (I(v) - I(u)) \,dxdy 
			\leqslant 0,
		\end{align}
		where equality takes place if and only if $|A_{+-}^*|_{2N}=0$ (and, equivalently, $|A_{-+}^*|_{2N}=0$).
		Assuming $|A_{+-}^*|_{2N}=0$, we get 
		\begin{equation}\label{eq:apm1}
			0 \geqslant u(\sigma_a(x)) > u(x)
			\quad \text{and} \quad
			u(\sigma_a(y)) < u(y) \leqslant 0
			\quad \text{for a.e.}~ (x,y) \in A_{+-}.
		\end{equation}
		Suppose now that there exists $x \in \Sigma_a^+$ such that
		\begin{equation}\label{eq:uvx1}
			u(\sigma_a(x)) > 0
			\quad \text{and} \quad
			u(\sigma_a(x)) > u(x).
		\end{equation}
		If there exists a point $y \in \Sigma_a^+$ such that $u(\sigma_a(y)) < u(y)$, 
		then $(x,y) \in A_{+-}$, 
		and hence the $2N$-measure of such points $(x,y)$ is zero in view of \eqref{eq:apm1}. 
		Thus, if \eqref{eq:uvx1} holds on a subset of $\Sigma_a^+$ of positive $N$-measure, then $u(\sigma_a(y)) \geqslant u(y)$ for a.e.\ $y \in \Sigma_a^+$, which contradicts the assumption $|A_{+-}|_{2N}>0$.
		Analogously, we get a contradiction if
		\begin{equation}\label{eq:uvx2}
			u(y) > 0
			\quad \text{and} \quad
			u(\sigma_a(y)) < u(y)
		\end{equation}
		hold on a subset of $\Sigma_a^+$ of positive $N$-measure.
		Therefore, combining these two facts, we conclude that for a.e.\ $x \in \Sigma_a^+$ such that $u(\sigma_a(x)) > 0$, we have $0 < u(\sigma_a(x)) \leqslant u(x)$, and for a.e.\ $y \in \Sigma_a^+$ such that $u(y) > 0$, we have $0 < u(y) \leqslant u(\sigma_a(y))$.
		Consequently, by redenoting $y$ to $x$, we deduce that for a.e.\ $x \in \Sigma_a^+$ such that $u(x)>0$ or $u(\sigma_a(x)) > 0$, we have $u(\sigma_a(x)) = u(x)$.
		This is exactly the case \ref{prop:polarization:3}.
		
		The inequality~\eqref{eq:improvement2} with equality cases can be established by noting that $u^- = -(-u)^+$. 
	\end{proof}
	
	In Figure~\ref{fig:proof1}, we depict a function $u: \mathbb{R} \to \mathbb{R}$ and its polarization $P_0u$ which deliver equality in \eqref{eq:improvement} under the assumption~\ref{prop:polarization:3} of Proposition~\ref{prop:polarization}, while neither the assumption~\ref{prop:polarization:1} nor~\ref{prop:polarization:2} holds.

	\begin{figure}[!h]
		\begin{center}
			\begin{subfigure}{0.49\textwidth}
				\includegraphics[width=\linewidth]{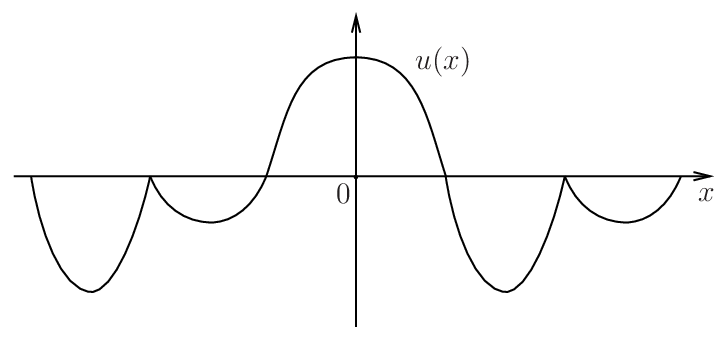}
			\end{subfigure}
			\hspace*{\fill}
			\begin{subfigure}{0.49\textwidth}
				\includegraphics[width=\linewidth]{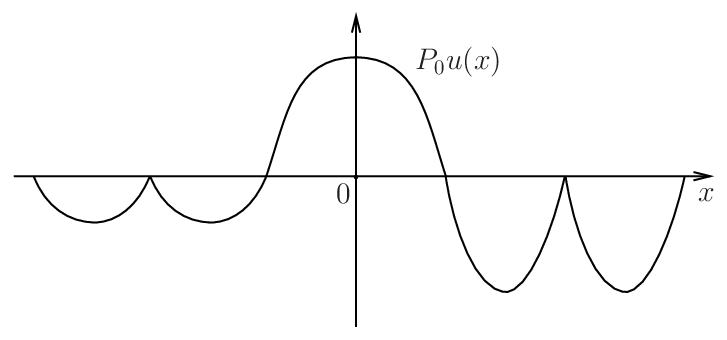}
			\end{subfigure}
		\end{center}
		\caption{A function $u: \mathbb{R} \to \mathbb{R}$ and its polarization $P_0u$ for which \eqref{eq:improvement} is an equality under the assumption~\ref{prop:polarization:3} of Proposition~\ref{prop:polarization}, but the assumptions~\ref{prop:polarization:1} and~\ref{prop:polarization:2} are not satisfied.} \label{fig:proof1}
	\end{figure}
	
	As a simple complementary fact to Proposition~\ref{prop:polarization}, we note that
	$$
	\langle D[u]_p^p, u^\pm \rangle \geqslant 0
	\quad 
	\text{for any}~ u \in \W,
	$$
	as it follows from the pointwise estimate (cf.~\cite[Eq.~(14)]{FranPal})
	$$
	|u(x)-u(y)|^{p-2}(u(x)-u(y))(u^\pm(x)-u^\pm(y)) 
	\geqslant 
	|u^\pm(x)-u^\pm(y)|^{p},
	\quad x,y \in \mathbb{R}^N.
	$$
	
	\begin{remark}\label{rem:equality}
		Summing \eqref{eq:improvement} and \eqref{eq:improvement2}, we obtain \eqref{eq:polineq1} and see that equality holds in $[P_a u]_p
		\leqslant 
		[u]_p$ 
		if and only either $u(x)=P_a u(x)$ for a.e.\ $x \in \mathbb{R}^N$ or $u(\sigma_a(x))=P_a u(x)$ for a.e.\ $x \in \mathbb{R}^N$. 
		In particular, when both \eqref{eq:improvement} and \eqref{eq:improvement2} are equalities, we have the same alternative. 
	\end{remark}
	
	\begin{remark}
		Using Lemma~\ref{lem:appendix:inequality}, one can explicitly estimate the deficit in \eqref{eq:improvement} and \eqref{eq:improvement2}.
		We also note that Proposition~\ref{prop:polarization} evidently holds for the polarization $\widetilde{P}_a$.
	\end{remark}
	
	\begin{remark}
		Since the proof of Proposition~\ref{prop:polarization} is largely based on the pointwise analysis, the particular choice of the kernel $|x-y|^{-(N+sp)}$ can be generalized to any 
		kernel $K(x,y)$ satisfying the following counterpart of \eqref{eq:7}:
		$$
		K(x,y)
		= 
		K(\sigma_a(x),\sigma_a(y))
		> 
		K(\sigma_a(x),y)
		= 
		K(x,\sigma_a(y))
		\quad \text{for every}~ x, y \in \Sigma_a^+.
		$$
	\end{remark}

	\medskip
	The following results are useful for the application of Proposition~\ref{prop:polarization} to functions from $\Wo$, cf.\ \cite[Corollary~5.1]{brocksol}.
	\begin{lemma}\label{lem:Pa-Sobolev}
		Let  $a \in \mathbb{R}$ and $u \in \Wo$ be a nonnegative function.
		Then $P_a u \in \widetilde{W}_0^{s,p}(P_a\Omega)$.
	\end{lemma}
	\begin{proof}
		Since $u$ is nonnegative, Lemma~\ref{lem:positive-part} gives a sequence $\{u_n\} \subset C_0^\infty(\Omega)$ of nonnegative functions converging to $u$ in $\Wo$. 
		It follows from \cite[Theorem~3.3 and Lemma~5.1]{brocksol} that each $P_a u_n$ is a nonnegative Lipschitz function with compact support in $P_a \Omega$. 
		Therefore, $P_a u_n \in \widetilde{W}_0^{s,p}(P_a \Omega)$ by Remark~\ref{rem:W-Lip}.
		
		By \eqref{eq:polineq1} and the convergence of $\{u_n\}$ in $\Wo$, we have $[P_a u_n]_p \leqslant [u_n]_p \leqslant C$ for some $C>0$ and all $n$.
		Thus, \cite[Theorem~2.7]{BLP} implies that $\{P_a u_n\}$ converges in $L^p(P_a\Omega)$ to a function $v \in \widetilde{W}_0^{s,p}(P_a\Omega)$, up to a subsequence. 
		On the other hand, 
		\cite[Theorem~3.1]{brocksol} yields
		$P_a u_n \to P_a u$ in $L^p(P_a \Omega)$.
		It is then clear that $v = P_au \in \widetilde{W}_0^{s,p}(P_a\Omega)$.
	\end{proof}

	It is not hard to see that Lemma~\ref{lem:Pa-Sobolev} is also valid for the polarization $\widetilde{P}_a$.
	In particular, applying Lemma~\ref{lem:Pa-Sobolev} to $u^+$ (with $P_a$) and to $-u^-$ (with $\widetilde{P}_a$), and using \eqref{eq:polarization_positive_negative2}, we get the following results.
	\begin{corollary}\label{cor:Pa-Sobolev}
		Let  $a \in \mathbb{R}$ and $u \in \Wo$.
		Then $P_a u \in \widetilde{W}_0^{s,p}(P_a\Omega \cup \widetilde{P}_a \Omega) = \widetilde{W}_0^{s,p}(\Omega \cup \sigma_a(\Omega))$.
	\end{corollary}

	\begin{lemma}\label{lem:Pa-Sobolev2}
		Let $\{a_n\} \subset \mathbb{R}$ be a sequence converging to $a \in \mathbb{R}$. 
		Let $u \in \Wo \cap C(\Omega)$ be a nonnegative function such that each $P_{a_n} (\mathrm{supp}\, u^+)$ is contained in $\Omega$.
		Then $P_{a_n} u \in \Wo$ for all $n$, and $P_{a} u \in \Wo$.
	\end{lemma}
	\begin{proof}
		In view of \eqref{eq:weak-pol} and \eqref{eq:polineq1}, we have $P_{a_n} u \in \W$ for any $n$.
		Since the closed set $P_{a_n} (\mathrm{supp}\, u^+)$ is a subset of $\Omega$ and $u$ is nonnegative, 
		we apply mollification arguments (see \cite[Lemma~11]{FSV}) to conclude that each $P_{a_n} u$ can be approximated by $C_0^\infty(\Omega)$-functions in the norm of $\W$.
		That is, $P_{a_n} u \in \Wo$ for any $n$.
		
		The inequality \eqref{eq:polineq1} shows that the sequence $\{P_{a_n} u\}$ is bounded in $\Wo$, and hence it converges in $L^p(\Omega)$ to some $v \in \Wo$, up to a subsequence (see \cite[Theorem~2.7]{BLP}).
		On the other hand, \cite[Lemma~5.2-1]{brocksol} gives
		$P_{a_n} u \to P_a u$ in $L^p(\Omega)$.
		Therefore, we conclude that $v = P_a u \in \Wo$.
	\end{proof}
	
	As above, it is not hard to observe that Lemma~\ref{lem:Pa-Sobolev2} remains valid for the polarization $\widetilde{P}_a$.

	\section[Characterizations]{Characterization of second eigenfunctions and LENS}\label{section:aux2}
	
	In this section, we characterize second eigenfunctions and least energy nodal solutions (LENS) of \eqref{eq:D} by certain integral inequalities. 
	These results will be needed for the application of Proposition~\ref{prop:polarization} in the proof of Theorem~\ref{thm:payne}.

	\subsection{Second eigenfunctions}\label{sec:secondeignvalue}
	We state three closely related results. 
	These results extend \cite[Lemma~2.1]{BBDG}, but the present arguments are different due to the general nonlinear settings; see also \cite{DP,DIS} for developments.
	
	Let us explicitly note that any second eigenfunction $u$ satisfies the following equalities:
	\begin{equation}\label{eq:1x}
		\lambda_2 \int_\Omega |u^+|^p \, dx 
		= 
		\frac{1}{p}\langle D[u]_p^p, u^+ \rangle
		\quad \text{and}\quad 
		\lambda_2 \int_\Omega |u^-|^p \, dx 
		= 
		\frac{1}{p}\langle D[u]_p^p, u^- \rangle.
	\end{equation}

	\begin{proposition}\label{prop:lambda2-characterization}
		Assume that there exists a function $v \in \Wo$ such that $v^\pm \not\equiv 0$  in $\Omega$ and 
		\begin{equation}\label{eq:1}
			\lambda_2 \int_\Omega |v^+|^p \, dx 
			\geqslant 
			\frac{1}{p}\langle D[v]_p^p, v^+ \rangle
			\quad \text{and}\quad 
			\lambda_2 \int_\Omega |v^-|^p \, dx 
			\geqslant 
			\frac{1}{p}\langle D[v]_p^p, v^- \rangle.
		\end{equation}
		Then $v$ is a second eigenfunction and equalities hold in \eqref{eq:1}.
	\end{proposition}
	\begin{proof}
		The first part of our arguments is reminiscent of the proof of \cite[Proposition~4.2]{BrPar}, where the authors establish that there is no eigenvalue between $\lambda_1$ and $\lambda_2$. 
		Since the assertion of the present proposition is of different nature and we use different notation, we provide details.
		Taking any $(\alpha,\beta) \in S^1$, multiplying the inequalities in \eqref{eq:1} by $|\alpha|^p$ and $|\beta|^p$, respectively, and then adding them, we get
		\begin{equation}\label{eq:proof:lambda2>1}
			\lambda_2 
			\geqslant
			\frac{	\frac{1}{p}\langle D[v]_p^p, |\alpha|^p v^+ + |\beta|^p v^-\rangle}
			{|\alpha|^p \int_\Omega |v^+|^p \, dx + |\beta|^p \int_\Omega |v^-|^p \, dx}.
		\end{equation}
		Denoting
		\begin{equation}\label{eq:lambda2-proof-char0}
			U(x,y) = v^+(x) - v^+(y)
			\quad \text{and} \quad 
			V(x,y) = -(v^-(x) - v^-(y)),
		\end{equation}
		we observe that
		\begin{equation}\label{eq:lambda2-proof-char1}
			v(x)-v(y) = (v^+(x) - v^+(y)) + (v^-(x) - v^-(y))
			=
			U(x,y) - V(x,y),
		\end{equation}
		and hence
		\begin{equation}\label{eq:proof:main-ineq0x}
			\frac{1}{p}\langle D[v]_p^p, |\alpha|^p v^+ + |\beta|^p v^-\rangle
			=
			\int_{\mathbb{R}^N} \int_{\mathbb{R}^N}
			\frac{|U-V|^{p-2} (U-V)(|\alpha|^p U - |\beta|^p V)}{|x-y|^{N+ps}} \,dxdy,
		\end{equation}
		cf.\ \eqref{eq:Dupxi}.
		It is not hard to see that $U V \leqslant 0$ a.e.\ in $\mathbb{R}^N \times \mathbb{R}^N$.
		Using the pointwise inequality from Lemma~\ref{lem:appendix:inequality} (which is essentially contained in the proof of \cite[Proposition~4.2]{BrPar}), we obtain
		\begin{equation}\label{eq:proof:main-ineq1}
			\int_{\mathbb{R}^N} \int_{\mathbb{R}^N}
			\frac{|U-V|^{p-2} (U-V)(|\alpha|^p U - |\beta|^p V)}{|x-y|^{N+ps}} \,dxdy
			\geqslant
			\int_{\mathbb{R}^N} \int_{\mathbb{R}^N}
			\frac{|\alpha U- \beta V|^{p}}{|x-y|^{N+ps}} \,dxdy
			=
			[\alpha v^+ +\beta v^-]_p^p.
		\end{equation}
		Thus, we deduce from \eqref{eq:proof:lambda2>1}, \eqref{eq:proof:main-ineq0x}, and \eqref{eq:proof:main-ineq1} that
		\begin{equation}\label{eq:proof:main-ineq1x}
			\lambda_2 \geqslant \frac{[\alpha v^+ +\beta v^-]_p^p}{|\alpha|^p \int_\Omega |v^+|^p \, dx + |\beta|^p \int_\Omega |v^-|^p \, dx}
			\quad \text{for any}~ (\alpha,\beta) \in S^1. 
		\end{equation}
		Consider a continuous odd function $h: S^1 \mapsto \Wo$ defined as
		$$
		h(\alpha,\beta) 
		= 
		\frac{\alpha v^+ + \beta v^-}{(|\alpha|^p \int_\Omega |v^+|^p \, dx + |\beta|^p \int_\Omega |v^-|^p \, dx)^\frac{1}{p}}.
		$$
		Clearly, we have $\|h(\alpha,\beta)\|_{p}=1$, 
		that is, $h: S^1 \mapsto \mathcal{S}$, where $\mathcal{S}$ is the unit $L^p(\Omega)$-sphere in $\Wo$ (see~\eqref{eq:Slp}), and the estimate \eqref{eq:proof:main-ineq1x} reads as
		\begin{equation}\label{eq:proof:main-ineq2}
			\frac{[\alpha v^+ +\beta v^-]_p^p}
			{|\alpha|^p \int_\Omega |v^+|^p \, dx + |\beta|^p \int_\Omega |v^-|^p \, dx} 
			\equiv
			[h(\alpha,\beta)]_p^p
			\leqslant 
			\lambda_2
			\quad \text{for any}~ (\alpha,\beta) \in S^1.
		\end{equation}	
		At the same time, the definition \eqref{eq:lambda2} of $\lambda_2$ implies that
		$$
		\lambda_2 \leqslant \max_{(\alpha,\beta) \in S^1} [h(\alpha,\beta)]_p^p.
		$$
		Thus, $\lambda_2 = [h(\alpha,\beta)]_p^p$ for some $(\alpha,\beta) \in S^1$.
		Applying \cite[Proposition~2.8]{cuesta}, we obtain the existence of $(\alpha_0,\beta_0) \in S^1$ such that $h(\alpha_0,\beta_0)$ is a 
		second eigenfunction, and hence so is $\alpha_0 v^+ + \beta_0 v^-$.
		Since any second eigenfunction is sign-changing (see \cite[Theorem~2.8~(iii)]{BrPar}), 
		we have $\alpha_0 \beta_0 >0$.		
		
		It remains to show that $\alpha_0=\beta_0$.
		Since $\lambda_2 = [h(\alpha_0,\beta_0)]_p^p$, we have equality in \eqref{eq:proof:main-ineq1} for $(\alpha,\beta)=(\alpha_0,\beta_0)$. 
		According to Lemma~\ref{lem:appendix:inequality}, this can happen if and only if $\alpha_0 = \beta_0$ or the set
		$$
		K = \{(x,y) \in \mathbb{R}^N \times \mathbb{R}^N:~ U(x,y) \cdot V(x,y) < 0\}
		$$
		has zero $2N$-measure. 
		Since $v^\pm \not\equiv 0$ in $\Omega$ by the assumption, there exist sets $K^\pm$ of positive $N$-measure such that $v^+ > 0$ in $K^+$ and $v^- < 0$ in $K^-$.
		Consequently, 
		$$
		U(x,y) \cdot V(x,y)
		=
		v^+(x) \cdot v^-(y) < 0
		\quad \text{for any}~ (x,y) \in K^+ \times K^-,
		$$ 
		and hence $K^+ \times K^- \subset K$, which yields $|K|_{2N} > 0$.
		Therefore, we must have $\alpha_0 = \beta_0$, that is, $v = v^+ + v^-$ is a second eigenfunction.
		As a consequence, a posteriori, equalities must hold in \eqref{eq:1}, cf.\ \eqref{eq:1x}.
	\end{proof}

	Proposition~\ref{prop:lambda2-characterization} implies the following result which will be convenient in applications.
	\begin{proposition}\label{prop:lambda2-main2}
		Let $u \in \Wo$ be a second eigenfunction.
		Assume that there exists a function $v \in \Wo$ such that $v^\pm \not\equiv 0$  in $\Omega$ and 
		\begin{equation}\label{eq:3x}
			\int_\Omega |v^\pm|^p \,dx
			\geqslant 			
			\int_\Omega |u^\pm|^p \,dx
			\quad \text{and} \quad 
			\langle D[v]_p^p, v^\pm \rangle
			\leqslant
			\langle D[u]_p^p, u^\pm \rangle.
		\end{equation}   
		Then $v$ is a second eigenfunction and equalities hold in \eqref{eq:3x}.
	\end{proposition}    
	\begin{proof}
		Since any second eigenfunction $u$ satisfies \eqref{eq:1x}, the assumptions \eqref{eq:3x} give
		$$
		\lambda_2 \int_\Omega |v^\pm|^p \,dx
		\geqslant
		\lambda_2 \int_\Omega |u^\pm|^p \,dx
		=
		\frac{1}{p}\langle D[u]_p^p, u^\pm \rangle
		\geqslant 
		\frac{1}{p}\langle D[v]_p^p, v^\pm \rangle.
		$$
		That is, $v$ satisfies the assumptions of Proposition~\ref{prop:lambda2-characterization}, and the conclusion follows.
	\end{proof}

	Another corollary of Proposition~\ref{prop:lambda2-characterization} is the following characterization of $\lambda_2$, cf.\ \cite[Remark~2.2]{BBDG} for the linear case $p=2$. We also refer to \cite{BrPar,Servadei2} for other characterizations of $\lambda_2$. 
	\begin{lemma}
		Let
		\begin{equation}\label{eq:lambda2-char}
			\mu_2 
			=
			\inf 
			\left\{
			\max
			\left\{
			\frac{\frac{1}{p}\langle D[v]_p^p, v^+ \rangle}{\int_\Omega |v^+|^p \, dx},
			\frac{\frac{1}{p}\langle D[v]_p^p, v^- \rangle}{\int_\Omega |v^-|^p \, dx}
			\right\}:~
			v \in \Wo, ~ v^\pm \not\equiv 0 \text{ in } \Omega
			\right\}.
		\end{equation}
		Then $\lambda_2 = \mu_2$ and any minimizer of $\mu_2$ is a second eigenfunction.
	\end{lemma}
	\begin{proof}
		Since any second eigenfunction $u$ satisfies \eqref{eq:1x}, we get $\mu_2 \leqslant \lambda_2$.			
		Suppose now, by contradiction, that $\mu_2 < \lambda_2$. 
		That is, there exists $v \in \Wo$ such that $v^\pm \not\equiv 0$ in $\Omega$ and
		\begin{equation}\label{eq:proof2}
			\mu_2 
			\leqslant
			\max
			\left\{
			\frac{\frac{1}{p}\langle D[v]_p^p, v^+ \rangle}{\int_\Omega |v^+|^p \, dx},
			\frac{\frac{1}{p}\langle D[v]_p^p, v^- \rangle}{\int_\Omega |v^-|^p \, dx}
			\right\} 
			< 
			\lambda_2.
		\end{equation}
		The second inequality in \eqref{eq:proof2} implies
		\begin{equation}\label{eq:proof3}
			\lambda_2 \int_\Omega |v^+|^p \, dx 
			\geqslant 
			\frac{1}{p}\langle D[v]_p^p, v^+ \rangle
			\quad \text{and}\quad 
			\lambda_2 \int_\Omega |v^-|^p \, dx 
			\geqslant 
			\frac{1}{p}\langle D[v]_p^p, v^- \rangle,
		\end{equation}
		at least one inequality being strict.
		However, this contradicts Proposition~\ref{prop:lambda2-characterization}.
		That is, we have $\mu_2 = \lambda_2$.
		In a similar way, any minimizer $v$ of $\mu_2$ satisfies the inequalities \eqref{eq:proof3}, and hence Proposition~\ref{prop:lambda2-characterization} shows that $v$ is a second eigenfunction.
	\end{proof}

	\subsection{LENS}
	In this section, we provide a result on the characterization of LENS, which has the same nature as Proposition~\ref{prop:lambda2-main2}.
	Consider the Nehari manifold associated with the problem \eqref{eq:D},
	\begin{equation*}
		\mathcal{N}
		=
		\{u \in \Wo \setminus \{0\}:~ \langle DE(u),u\rangle = 0\},
	\end{equation*}
	and the following subset of $\mathcal{N}$ (a nodal Nehari set) which contains all nodal solutions of \eqref{eq:D}:
	\begin{equation*}
		\mathcal{M}
		=
		\{u \in \Wo:~ u^\pm \not\equiv 0 ~\text{in}~ \Omega,~ \langle DE(u),u^+\rangle = \langle DE(u),u^-\rangle = 0\},
	\end{equation*} 
	cf.\ \eqref{eq:DEu}.
	It is known that, under the assumption \ref{Fs}, any minimizer of the problem
	\begin{equation}\label{eq:m}
		m = \inf\{E(u):~ u \in \mathcal{M}\}
	\end{equation}
	is a LENS, see \cite[Lemma~4.7]{CNW} and also comments and references provided in Section~\ref{sec:intro}.

	\begin{proposition}\label{prop:altern-char-lens}
		Let $u$ be a LENS.
		Assume that there exists a function $v \in \Wo$ such that $v^\pm \not\equiv 0$  in $\Omega$ and 
		\begin{equation}\label{eq:2}
			\int_\Omega F(v) \,dx
			\geqslant
			\int_\Omega F(u) \,dx,
			\quad
			\int_\Omega f(v^\pm)v^\pm \, dx
			=
			\int_\Omega f(u^\pm)u^\pm \, dx,
			\quad 
			\langle D[v]_p^p, v^\pm \rangle
			\leqslant
			\langle D[u]_p^p, u^\pm \rangle.
		\end{equation}		  
		Then $v$ is a LENS and equalities hold in \eqref{eq:2}.
	\end{proposition}
	\begin{proof}
		Some parts of our arguments are reminiscent of those from the proof of \cite[Lemma~4.5]{CNW}, where the authors obtain the attainability of $m$ defined  in \eqref{eq:m}.
		We provide details since our statement is different and we employ other notation. 			
		By \cite[Lemma~4.4]{CNW}, there exist unique positive numbers $t_+, t_-$ such that $t_+ v^+ + t_- v^- \in \mathcal{M}$, which reads as
		\begin{align}
			\label{eq:lens:1}
			\frac{1}{p} \,\langle D[t_+ v^+ + t_- v^-]_p^p, t_+ v^+ \rangle 
			&=
			\int_\Omega f(t_+ v^+ + t_- v^-) t_+ v^+ \, dx
			\equiv
			\int_\Omega f(t_+ v^+) t_+ v^+ \, dx,\\
			\label{eq:lens:2}
			\frac{1}{p} \,\langle D[t_+ v^+ + t_- v^-]_p^p, t_- v^- \rangle 
			&=
			\int_\Omega f(t_+ v^+ + t_- v^-) t_- v^- \, dx
			\equiv
			\int_\Omega f(t_- v^-) t_- v^- \, dx.
		\end{align}
		Let us show that $t_\pm \in (0,1]$.
		Assume, without loss of generality, that $t_- \leqslant t_+$.
		In view of the homogeneity of the left-hand side of \eqref{eq:lens:1}, we rewrite it as 
		\begin{equation}
			\label{eq:prop:proof:lens:2}
			\frac{1}{p} \,\langle D\Big[v^+ + \frac{t_-}{t_+} v^-\Big]_p^p, v^+ \rangle 
			=
			\int_\Omega f(t_+ v^+) \, t_+^{1-p} v^+ \,dx.
		\end{equation}
		Denoting, as in \eqref{eq:lambda2-proof-char0}, 
		$$
		U(x,y) = v^+(x) - v^+(y)
		\quad \text{and} \quad
		V(x,y) = -(v^-(x) - v^-(y)),
		$$
		and observing, similarly to \eqref{eq:lambda2-proof-char1}, that 
		$$
		(v^+(x) + sv^-(x)) - (v^+(y) + sv^-(y))
		=
		U(x,y) - sV(x,y)
		\quad \text{for any}~ s \in \mathbb{R},
		$$
		and $U V \leqslant 0$ a.e.\ in $\mathbb{R}^N \times \mathbb{R}^N$, 
		we apply Lemma~\ref{lem:appendix:inequality2} with $s=t_-/t_+ \in (0,1]$ and get
		\begin{equation}\label{eq:prop:proof:lens:x1}
			\langle D\Big[v^+ + \frac{t_-}{t_+} v^-\Big]_p^p, v^+ \rangle 
			\leqslant 
			\langle D[v^+ + v^-]_p^p, v^+ \rangle 
			\equiv 
			\langle D[v]_p^p, v^+ \rangle. 
		\end{equation}
		On the other hand, since $u$ is a solution of \eqref{eq:D}, we use the second and third assumptions from \eqref{eq:2} to obtain
		\begin{equation}\label{eq:prop:proof:lens:3}
			\frac{1}{p} \,\langle D[v]_p^p, v^+ \rangle  
			\leqslant
			\frac{1}{p} \,\langle D[u]_p^p, u^+ \rangle  
			=
			\int_\Omega f(u) u^+ \, dx
			\equiv
			\int_\Omega f(u^+) u^+ \, dx
			=
			\int_\Omega f(v^+) v^+ \, dx. 
		\end{equation}
		Combining \eqref{eq:prop:proof:lens:2}, \eqref{eq:prop:proof:lens:x1}, and \eqref{eq:prop:proof:lens:3}, we derive
		$$
		\int_\Omega 
		\left(
		\frac{f(v^+)}{(v^+)^{p-1}} 
		-
		\frac{f(t_+ v^+)}{(t_+ v^+)^{p-1}} 
		\right)
		(v^+)^p \,dx
		\geqslant 0.
		$$
		Since $z \mapsto f(z)/z^{p-1}$ is increasing in $(0,+\infty)$ by the assumptions \ref{Fs}~\ref{Fs-b}, we conclude that $t_+ \leqslant 1$, and hence $t_- \leqslant t_+ \leqslant 1$.
		
		Consider now a function $G$ defined as $G(z) = f(z)z - pF(z)$ and note that $G(0)=0$.
		Since $t_+ v^+ + t_- v^- \in \mathcal{M}$, we have
		\begin{align}
			m
			&\leqslant 
			E(t_+ v^+ + t_- v^-)
			\\
			&=
			E(t_+ v^+ + t_- v^-)
			-
			\frac{1}{p}\langle 
			DE(t_+ v^+ + t_- v^-), 
			t_+ v^+ + t_- v^-
			\rangle
			\\
			&=
			\label{eq:prop:lens:1}
			\frac{1}{p}
			\int_\Omega G(t_+ v^+ + t_- v^-) \,dx
			=
			\frac{1}{p}
			\int_\Omega G(t_+ v^+)\,dx
			+
			\frac{1}{p}
			\int_\Omega G(t_- v^-)\,dx.
		\end{align}
		The assumptions \ref{Fs}~\ref{Fs-b} imply that $G$ is 
		decreasing in $(-\infty,0)$, 
		increasing in $(0,+\infty)$, and nonnegative in $\mathbb{R}$.
		Therefore, since $t_\pm \in (0,1]$, we get
		\begin{equation}\label{eq:prop:lens:2}
			m
			\leqslant
			\frac{1}{p}
			\int_\Omega G(t_+ v^+)\,dx
			+
			\frac{1}{p}
			\int_\Omega G(t_- v^-)\,dx
			\leqslant  
			\frac{1}{p}
			\int_\Omega G(v^+)\,dx
			+
			\frac{1}{p}
			\int_\Omega G(v^-)\,dx
			=
			\frac{1}{p}
			\int_\Omega G(v)\,dx.
		\end{equation}
		In view of the first and second assumptions from \eqref{eq:2}, 
		we obtain
		\begin{equation}\label{eq:prop:lens:3}
			m
			\leqslant  
			\frac{1}{p}
			\int_\Omega G(v)\,dx
			\leqslant
			\frac{1}{p}
			\int_\Omega G(u)\,dx  
			=
			E(u) - \frac{1}{p} \langle DE(u),u \rangle
			=
			E(u)  
			= 
			m.
		\end{equation}
		That is, equalities hold in \eqref{eq:prop:lens:1}, \eqref{eq:prop:lens:2}, \eqref{eq:prop:lens:3}, which yields
		$t_\pm = 1$ and $v \in \mathcal{M}$ is a minimizer of $E$ over $\mathcal{M}$.
		Moreover, equalities take place in \eqref{eq:2}.
		Consequently, by \cite[Lemma~4.7]{CNW}, $v$ is a least energy nodal solution of \eqref{eq:D}.
	\end{proof}

	\begin{remark}
		Let us note that, in general, the equalities $\int_\Omega f(v^\pm)v^\pm \, dx = \int_\Omega f(u^\pm)u^\pm \, dx$ in \eqref{eq:2} do not imply that $\int_\Omega F(v) \,dx = \int_\Omega F(u) \,dx$, and the reverse implication cannot be guaranteed either.
		This can be seen by considering the model case $f(z) = |z|^{\alpha-2}z + |z|^{\beta-2}z$ for $p<\alpha<\beta<p_s^*$ and with sign-changing functions $u,v \in \Wo$ satisfying, e.g.,
		\begin{align*}
			\int_\Omega |v^\pm|^\alpha \,dx = 1,
			\quad 
			\int_\Omega |v^\pm|^\beta \,dx = 2,
			\quad
			\int_\Omega |u^\pm|^\alpha \,dx = 2,
			\quad
			\int_\Omega |u^\pm|^\beta \,dx = 1.
		\end{align*}
		Consequently, in general, the first two assumptions in \eqref{eq:2} are independent from each other.  
	\end{remark}

	\begin{remark}\label{rem:Fs}
		The proof of Proposition~\ref{prop:altern-char-lens} relies on the results from \cite[Section~4]{CNW}.
		If these results are valid under weaker (or just different) assumptions on $f$ than \ref{Fs} (see, e.g., the assumptions in \cite{GTZ,GYZ,luo,tengwangwang} for the case $p=2$ and \cite{WZ} for the case $p>1$), then so does Proposition~\ref{prop:altern-char-lens}, and hence \ref{Fs} can be changed accordingly.
	\end{remark}

	\section{Proof of Theorem~\ref{thm:payne}}\label{section:proof}
	
	Let $u \in \Wo$ be either a second eigenfunction or LENS of \eqref{eq:D}. 
	Suppose, contrary to the statement of Theorem~\ref{thm:payne}, that $u$ does not change sign in a neighborhood of $\partial \Omega$.
	Without loss of generality,  
	let $\mathrm{supp}\, u^- \subset \Omega$, so that $u \geqslant 0$ in this neighborhood. 
	
	Since $\Omega$ is Steiner symmetric with respect to the hyperplane $H_0$, we have $P_0 \Omega = \Omega$ and $\widetilde{P}_0 \Omega = \Omega$, see, e.g., \cite[Lemma~2.2]{BK1}.
	Therefore, Corollary~\ref{cor:Pa-Sobolev} gives 
	$P_0 u \in \Wo$. 
	Combining the inequalities from Proposition~\ref{prop:polarization} and equalities \eqref{eq:weak-pol} with either Proposition~\ref{prop:lambda2-main2} (when $u$ is a second eigenfunction) or Proposition~\ref{prop:altern-char-lens} (when $u$ is a LENS), we deduce that $P_0 u$ is also either a second eigenfunction or LENS. 
	In particular, equalities hold in \eqref{eq:improvement}, \eqref{eq:improvement2}, which implies that either $P_0 u(x) = u(x)$ for all $x \in \mathbb{R}^N$ or $P_0 u(x) = u(\sigma_0(x))$ for all $x \in \mathbb{R}^N$, see Remark~\ref{rem:equality}.
	Assume, without loss of generality, that $P_0 u = u$ in $\mathbb{R}^N$.	
	In particular, this yields 
	\begin{equation}
		\label{eq:uu1}
		u(x) \leqslant u(\sigma_0(x))
		\quad \text{for any}~ 
		x \in \Sigma_0^+.
	\end{equation}
	
	Let us now define 
	\begin{equation}\label{eq:d1}
		d_1 
		= 
		\sup\{ 
		t \geqslant 0:~ \mathrm{supp}\, u^- + t e_1 \subset \Omega
		\}.
	\end{equation}
	Our assumption $\mathrm{supp}\, u^- \subset \Omega$ gives $d_1 > 0$. 
	We fix $a = d_1 / 2$ and consider the polarization $P_a u$. 
	We see that $\mathrm{supp}\,P_{a_n}u^- =\widetilde{P}_{a_n} (\mathrm{supp}\, u^-) \subset \Omega$ for any sequence $a_n \nearrow a$, and $a$ is the supremum among polarization parameters with this set inclusion property, see Figure~\ref{fig:proof}.
	Therefore, applying Lemma~\ref{lem:Pa-Sobolev2} to $-u^-$ (with $\widetilde{P}_a$), we get $P_a u^- \equiv - \widetilde{P}_a(-u^-) \in \Wo$.
	On the other hand, again by \cite[Lemma~2.2]{BK1}, we have $P_a \Omega =\Omega$, and hence Lemma~\ref{lem:Pa-Sobolev} applied to $u^+$ gives $P_a u^+ \in \Wo$.
	Thus, we conclude that $P_a u \in \Wo$ and $\mathrm{supp}\,P_au^-$ touches $\partial \Omega \cap \Sigma_a^+$. 
	
	\begin{figure}[!h]
		\begin{center}
			\includegraphics[width=0.8\linewidth]{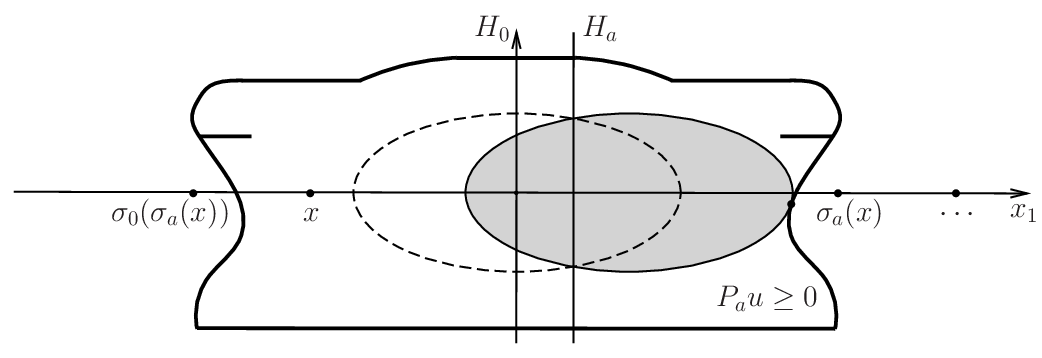}
		\end{center}
		\caption{The gray oval is $\mathrm{supp}\,P_au^-$, and the dashed oval is the boundary of $\mathrm{supp}\,u^-$.} 
		\label{fig:proof}
	\end{figure}
	
	As above, a combination of Proposition~\ref{prop:polarization}, equalities \eqref{eq:weak-pol}, and either Proposition~\ref{prop:lambda2-main2} or Proposition~\ref{prop:altern-char-lens} guarantees that $P_a u$ must be either a second eigenfunction or LENS, and equalities hold in \eqref{eq:improvement}, \eqref{eq:improvement2}.
	Since $u \geqslant 0$ in a neighborhood of $\partial\Omega$ but $\mathrm{supp}\,P_au^-$ touches $\partial \Omega \cap \Sigma_a^+$, we conclude that $P_a u(x) \neq u(x)$ for some $x \in \Sigma_a^+$.
	Therefore,  Proposition~\ref{prop:polarization} implies that $P_a u(x) = u(\sigma_a(x))$ \textit{for all} $x \in \mathbb{R}^N$, see Remark~\ref{rem:equality}.
	(This is the main place where the characterization of equality cases in Proposition~\ref{prop:polarization} is used.)
	In particular, this yields 
	\begin{equation}
		\label{eq:uu2}
		u(x) \leqslant u(\sigma_a(x))
		\quad \text{for any}~ 
		x \in \Sigma_a^-.
	\end{equation}
	
	Let us now obtain a contradiction from \eqref{eq:uu1} and \eqref{eq:uu2}. 
	Take any $x \in \Sigma_a^-$ such that $u(x) > 0$.
	Then \eqref{eq:uu2} gives $u(\sigma_a(x))>0$, where $\sigma_a(x) \in \Sigma_a^+$.
	We always have $\Sigma_a^+ \subset \Sigma_0^+$.
	Therefore, \eqref{eq:uu1} applied to $\sigma_a(x)$ gives $u(\sigma_0(\sigma_a(x)))>0$, where $\sigma_0(\sigma_a(x)) \in \Sigma_0^-$.
	We always have $\Sigma_0^- \subset \Sigma_a^-$.
	Hence, we again apply \eqref{eq:uu2}, etc.
	The consecutive application of \eqref{eq:uu1} and \eqref{eq:uu2} leads to the infinite chain of inequalities 
	\begin{equation}\label{eq:prof:teorx1}
		0 < u(x) \leqslant u(\sigma_a(x)) \leqslant u(\sigma_0(\sigma_a(x)))
		\leqslant \ldots \leqslant 
		u(\sigma_a(\sigma_0(\dots \sigma_a(\sigma_0(x)))))
		\leqslant \ldots
	\end{equation}
	In particular, recalling that $u \in \Wo$, we see that $\sigma_a(\sigma_0(\dots \sigma_a(\sigma_0(x)))) \in \Omega$ for any number of iterations.
	However, it is not hard to observe that the point
	$\sigma_a(\sigma_0(\dots \sigma_a(\sigma_0(x))))$ 
	moves to infinity along the $x_1$-axis as the number of iterations grows.
	Indeed, the first coordinates of these points satisfy the following relations (see Figure~\ref{fig:proof}):
	\begin{align*}
		(\sigma_a(x))_1 
		&= 
		2a - x_1,\\
		(\sigma_0(\sigma_a(x)))_1 
		&= 
		- (\sigma_a(x))_1
		=
		x_1 - 2a,\\
		(\sigma_a(\sigma_0(\sigma_a(x)))_1
		&=
		2a - (\sigma_0(\sigma_a(x)))_1 
		=
		4a - x_1,\\
		&\cdots
	\end{align*}
	Since $\Omega$ is bounded and $u = 0$ in $\mathbb{R}^N \setminus \Omega$, we get a contradiction.
	
	Notice that the initial choice $x \in \Sigma_0^-$ for the assumption $u(x)>0$ is not restrictive.
	Indeed, if $x \in \Sigma_0^+$ is such that $u(x)>0$, then \eqref{eq:uu1} gives $u(\sigma_0(x)) > 0$ and $\sigma_0(x) \in  \Sigma_0^-$ and we can redenote $\sigma_0(x)$ by $x$, while if $x \in H_0$ is such that $u(x)>0$, then we can shift $x$ it to the left due to the continuity of $u$.
	This finishes the proof.
	\qed

	\bigskip
	\begin{remark}
		The polarization arguments in the proof of Theorem~\ref{thm:payne} do not involve any Hopf's type information about $u$ and require no regularity of $\partial \Omega$ (unlike the proof of \cite[Theorem~1.2]{BK1} about the local nonlinear case), and they do not use a careful analysis of the structure of $P_a u$ (unlike the proof of \cite[Theorem~1.1]{BBDG} about the nonlocal linear case in the ball). 
		The additional constructions from \cite{BBDG,BK1} are ``substituted'' by the characterization of equality cases in Proposition~\ref{prop:polarization}. 
		
		However, it is hard to adapt a similar idea to the local case (e.g., with the aim of weakening regularity assumptions on $\partial \Omega$ imposed in \cite[Theorem~1.2]{BK1}), since 
		local counterparts of the inequalities  \eqref{eq:improvement} and \eqref{eq:improvement2} from Proposition~\ref{prop:polarization} are \textit{always equalities}, see \cite[Lemma~2.3]{BartschWethWillem}. 
		In particular, in the local case, we cannot guarantee that $P_a u(x) = u(\sigma_a(x))$ for all $x \in \mathbb{R}^N$.
	\end{remark}
	
	\begin{remark}
		In the proof of Theorem~\ref{thm:payne}, the boundedness of $\Omega$ can be substituted by the boundedness of $N$-measure of $\Omega$, by noting that the process of consecutive reflections with respect to $H_0$ and $H_a$ (see \eqref{eq:uu1} and \eqref{eq:uu2}) ``pushes'' any set to infinity along the $x_1$-axis.
	\end{remark}
	
	\begin{remark}
		Theorem~\ref{thm:payne} and all the results of Sections~\ref{section:aux}, \ref{section:aux2} remain valid if we substitute the space $\Wo$ by
		\begin{equation}\label{eq:Xsp}
			X_0^{s,p}(\Omega) 
			=
			\{
			u \in \W:~ u=0 ~\text{a.e. in}~ \mathbb{R}^N \setminus \Omega
			\},
		\end{equation}
		provided $\Omega$ supports the compactness of the embedding $X_0^{s,p}(\Omega) \hookrightarrow L^p(\Omega)$. 
		(Lemmas~\ref{lem:Pa-Sobolev} and~\ref{lem:Pa-Sobolev2} follow directly from the definition of $X_0^{s,p}(\Omega)$.)
		It is not hard to see that $\Wo \subset X_0^{s,p}(\Omega)$.
		Moreover, equality holds if $\partial \Omega$ is sufficiently regular, see, e.g., \cite{FSV}.
		But, in general, the space $X_0^{s,p}(\Omega)$ is rougher than $\Wo$ since it is not sensitive to perturbations of $\Omega$ by sets of zero $N$-measure (e.g., ``cuts'' in $\Omega$ are invisible for $X_0^{s,p}(\Omega)$). 
	\end{remark}

	\begin{remark}
		The proof of Theorem~\ref{thm:payne} justifies a stronger assertion than Theorem~\ref{thm:payne}. 
		Assume, for simplicity, that $\Omega$ is a bounded open set with continuous boundary in the sense of \cite[Definition~4]{FSV}.
		Let us decompose $\partial \Omega$ in three parts - the left ``lid'' $L$, right ``lid'' $R$, and cylindrical part $C$ parallel to the $x_1$-axis, as follows.
		Let us take any open segment $l \subset \Omega$ parallel to the $x_1$-axis, symmetric with respect to $H_0$, and such that end-points of $l$ lie on $\partial \Omega$. 
		The sets $L$ and $R$ are the unions of left and right end-points of such segments, respectively, and $C = \partial \Omega \setminus (L \cup R)$, cf.\ Figure~\ref{fig:improv}.
		Let $u \in \Wo$ be a second eigenfunction or LENS of \eqref{eq:D}.
		Then $u$ necessarily satisfies at least one of the following two properties:
		\begin{enumerate}
			\item[1)] $\mathrm{supp}\,u^+ \cap \overline{L} 
			\neq \emptyset~$
			and 
			$~\mathrm{supp}\,u^- \cap \overline{R}
			\neq \emptyset$,
			\item[2)] $\mathrm{supp}\,u^+ \cap \overline{R}
			\neq \emptyset~$ 
			and 
			$~\mathrm{supp}\,u^- \cap \overline{L}
			\neq \emptyset$.
		\end{enumerate}
		To establish this assertion, 
		the proof of Theorem~\ref{thm:payne} is repeated almost verbatim. 
		Notice that, under the current assumptions on $\Omega$, the result of Lemma~\ref{lem:Pa-Sobolev2} remains valid if we allow $P_{a_n}(\mathrm{supp}\,u^+) \subset \overline{\Omega}$, as it follows from the equality $\Wo = X_0^{s,p}(\Omega)$, see \cite[Theorem~6]{FSV}.
		We omit further details.
		
		In Figure~\ref{fig:improv1x} we depict a hypothetical behavior of $u$ which is ruled out by this assertion and not by Theorem~\ref{thm:payne}. 
		
		\begin{figure}[!ht]
			\begin{center}
				\begin{subfigure}{0.49\textwidth}
					\includegraphics[width=\linewidth]{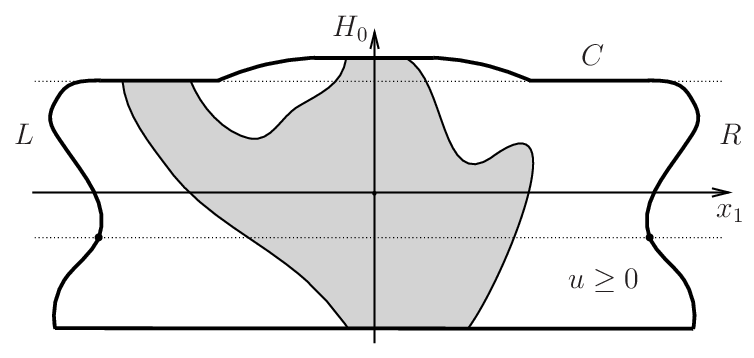}
					\caption{}
					\label{fig:improv1x}
				\end{subfigure}
				\hspace*{\fill}
				\begin{subfigure}{0.49\textwidth}
					\includegraphics[width=\linewidth]{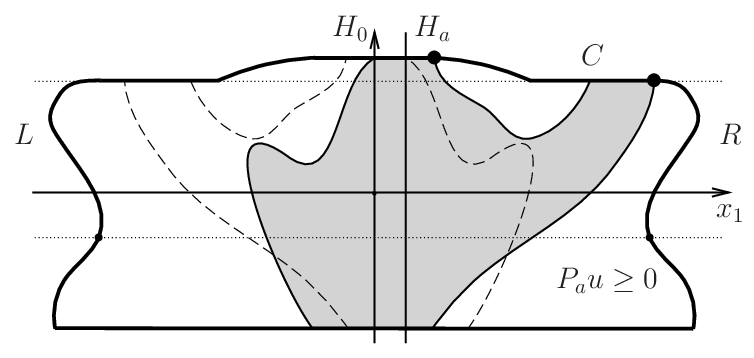}
					\caption{}
					\label{fig:improv2x}
				\end{subfigure}
			\end{center}
			\caption{(a): $u$ is positive in the white part and negative in the gray part (and hence $u=0$ on the boundary of the gray part), that is, $\mathrm{supp}\,u^- \cap (\overline{L} \cup \overline{R}) = \emptyset$. (b): polarization of $u$ with respect to $H_a$ for a maximal value of $a$, such that $\mathrm{supp}\,P_au^-$ touches $\overline{R}$ at two bold dots.} 
			\label{fig:improv}
		\end{figure}
	\end{remark}

	\appendix
	\section{Auxiliary results}\label{section:appendix}
	In this section, we collect a few technical results used in the proofs above.
	We start with a four-point inequality needed for Proposition~\ref{prop:polarization}.
	Let a function $J: \mathbb{R}^2 \mapsto \mathbb{R}$ be defined as in \eqref{eq:J}, i.e.,
	$$
	J(\alpha,\beta) = |\alpha - \beta|^{p-2}(\alpha - \beta)(\alpha^+ - \beta^+).
	$$
	Recall that $\alpha^+ = \max\{\alpha,0\}$.
	Also, we denote by $\theta: \mathbb{R} \mapsto \mathbb{R}$ the Heaviside function and we assume $\theta(0)=0$, for definiteness.
	Rewriting $J$ in terms of $\theta$, we have
	$$
	J(\alpha,\beta) = |\alpha - \beta|^{p-2}(\alpha - \beta)(\theta(\alpha)\alpha - \theta(\beta) \beta).
	$$
	\begin{lemma}\label{lem:four_point_lemma}
		Let $p>1$. 
		Assume that $a < A$ and $b < B$. 
		Then
		\begin{align}
			-(p-1) \, \max \{1,p-1\}
			&\int_a^A \int_b^B  
			|\alpha-\beta|^{p-2} \big(\theta (\alpha) + \theta(\beta)\big)
			\,d\beta d\alpha
			\\
			&\leqslant
			J(A,B) - J(a,B) - J(A,b) + J(a,b) 
			\\
			\label{eq:lem1}
			&\leqslant
			-(p-1) \, \min \{1,p-1\}
			\int_a^A \int_b^B  
			|\alpha-\beta|^{p-2} \big(\theta (\alpha) + \theta(\beta)\big)
			\,d\beta d\alpha.~~~~
		\end{align}
		In particular, 
		\begin{equation}\label{eq:J1}
			J(A,B) - J(a,B) - J(A,b) + J(a,b) \leqslant 0,
		\end{equation}
		and equality takes place in \eqref{eq:J1} if and only if $A \leqslant 0$ and $B \leqslant 0$.
	\end{lemma}
	\begin{proof}
		We start with \textit{formal} computations, assuming that all operations are allowed.
		Observe that
		\begin{equation}
			\label{eq:four_point_integral_representation} 
			J(A,B) - J(a,B) - J(A,b) + J(a,b) = \int_a^A \int_b^B \frac{\partial^2 J}{\partial \alpha \partial \beta}(\alpha, \beta) d\beta d\alpha.   
		\end{equation}
		Differentiating $J$, we obtain
		\begin{equation}
			\frac{\partial J}{\partial \alpha}(\alpha, \beta) = (p-1) |\alpha - \beta|^{p-2} (\theta(\alpha)\alpha - \theta(\beta) \beta) 
			+ 
			|\alpha - \beta|^{p-2}(\alpha - \beta) \theta(\alpha)
		\end{equation}
		and
		\begin{align}
			\frac{\partial^2 J}{\partial \alpha \partial \beta}(\alpha, \beta)
			= 
			&- (p-1)(p-2) |\alpha - \beta|^{p-4}(\alpha - \beta)(\theta(\alpha)\alpha - \theta(\beta) \beta) 
			\\
			&- (p-1) |\alpha - \beta|^{p-2} \theta(\beta) 
			- (p-1) |\alpha - \beta|^{p-2}\theta(\alpha) 
			\\
			\label{eq:JJJJ2}
			&= - (p-1) |\alpha - \beta|^{p-2} \left [ (p-2) \frac {\theta(\alpha)\alpha - \theta(\beta) \beta}{\alpha - \beta} + \theta(\alpha) + \theta(\beta) \right ].
		\end{align}
		For $\alpha \neq \beta$, we have
		\begin{equation}
			0 \leqslant \frac {\theta(\alpha)\alpha - \theta(\beta) \beta}{\alpha - \beta} \leqslant \max \lbrace \theta(\alpha), \theta(\beta) \rbrace \leqslant \theta(\alpha) + \theta(\beta),
		\end{equation}
		and therefore the expression in the square brackets in \eqref{eq:JJJJ2} can be estimated as follows:
		\begin{equation}
			(\theta(\alpha) + \theta(\beta)) 
			\leqslant (p-2) \frac {\theta(\alpha)\alpha - \theta(\beta) \beta}{\alpha - \beta} + \theta(\alpha) + \theta(\beta) \leqslant (p-1) (\theta(\alpha) + \theta(\beta))
		\end{equation}
		for $p \geqslant 2$, and
		\begin{equation}
			(p-1) (\theta(\alpha) + \theta(\beta)) 
			\leqslant (p-2) \frac {\theta(\alpha)\alpha - \theta(\beta) \beta}{\alpha - \beta} + \theta(\alpha) + \theta(\beta) \leqslant (\theta(\alpha) + \theta(\beta))
		\end{equation}
		for $p \in (1,2)$.		
		This formally yields the desired inequalities \eqref{eq:lem1}.
		
		To conclude the proof, let us substantiate the formal calculations. For brevity,
		we denote the integral term in \eqref{eq:lem1} as $I(a,A,b,B)$ and by $C_1$, $C_2$ the corresponding constants on the left- and right-hand sides.
		With these notation, \eqref{eq:lem1} reads as
		\begin{equation}\label{eq:lem12}
			C_1 \, I(a,A,b,B)
			\leqslant
			J(A,B) - J(a,B) - J(A,b) + J(a,b)
			\leqslant
			C_2 \, I(a,A,b,B).
		\end{equation}
		Let us rewrite $J$ as
		\begin{equation}
			J(\alpha,\beta) = 
			\begin{cases}
				|\alpha - \beta|^p, 
				&\alpha, \beta \geqslant 0,\\
				|\alpha-\beta|^{p-2}(\alpha-\beta)\alpha, 
				&\alpha \geqslant 0 > \beta,\\
				|\alpha-\beta|^{p-2}(\alpha-\beta)(-\beta), 
				&\alpha < 0 \leqslant \beta,\\
				0, 
				&\alpha, \beta \leqslant 0.
			\end{cases}
		\end{equation}
		We see that $J$ is continuous in $\mathbb{R}^2$ and there are only three lines on which the smoothness of $J$ can be compromised: $\alpha=0$, $\beta = 0$, and $\alpha=\beta \geqslant 0$. Outside of these three lines, $J$ is $C^\infty$-smooth. On the diagonal except the origin, i.e., when $\alpha = \beta \neq 0$, we see that $J$ is at least $C^1$-smooth (since $p>1$), with first partial derivatives being absolutely continuous. 
		
		Assume that $\alpha<0<A$ and $\beta<0<B$, and take any sufficiently small $\varepsilon > 0$. 
		Thanks to the regularity of $J$ discussed above, all our formal calculations are rigorous on each of the following subrectangles of $[a,A]\times [b,B]$:
		$$
		[\varepsilon,A]\times[\varepsilon,B],
		\quad 
		[a,-\varepsilon]\times[\varepsilon,B],
		\quad 
		[a,-\varepsilon]\times[b,-\varepsilon],
		\quad
		[\varepsilon,A]\times[b,-\varepsilon].
		$$
		(In other words, we deleted from the rectangle $[a,A]\times [b,B]$ narrow strips around the lines $a=0$ and $b=0$.)
		This leads to the validity of the estimates of the type \eqref{eq:lem12} on every subrectangle:
		\begin{align*}
			C_1 \, I(\varepsilon,A,\varepsilon,B)
			\leqslant
			J(A,B) - J(\varepsilon,B) &- J(A,\varepsilon) + J(\varepsilon,\varepsilon)
			\leqslant
			C_2 \, I(\varepsilon,A,\varepsilon,B),\\
			C_1 \, I(a,-\varepsilon,\varepsilon,B)
			\leqslant
			J(-\varepsilon,B) - J(a,B) &- J(-\varepsilon,\varepsilon) + J(a,\varepsilon)
			\leqslant
			C_2 \, I(a,-\varepsilon,\varepsilon,B),\\
			C_1 \, I(a,-\varepsilon,b,-\varepsilon)
			\leqslant
			J(-\varepsilon,-\varepsilon) - J(a,-\varepsilon) &- J(-\varepsilon,b) + J(a,b)
			\leqslant
			C_2 \, I(a,-\varepsilon,b,-\varepsilon),\\
			C_1 \, I(\varepsilon,A,b,-\varepsilon)
			\leqslant
			J(A,-\varepsilon) - J(\varepsilon,-\varepsilon) &- J(A,b) + J(\varepsilon,b)
			\leqslant
			C_2 \, I(\varepsilon,A,b,-\varepsilon).
		\end{align*}
		Let us now pass to the limit as $\varepsilon \to 0$. Since $J$ is continuous, we can do it in all middle terms.
		Noticing that the singularity $|\alpha-\beta|^{p-2}$ in $I(a,A,b,B)$ is integrable, we apply the dominated convergence theorem to deduce that the passage to the limit is allowed on the left- and right-hand sides, as well.
		Thus, we get
		\begin{align*}
			C_1 \, I(0,A,0,B)
			\leqslant
			J(A,B) - J(0,B) &- J(A,0) + J(0,0)
			\leqslant
			C_2 \, I(0,A,0,B),\\
			C_1 \, I(a,0,0,B)
			\leqslant
			J(0,B) - J(a,B) &- J(0,0) + J(a,0)
			\leqslant
			C_2 \, I(a,0,0,B),\\
			C_1 \, I(a,0,b,0)
			\leqslant
			J(0,0) - J(a,0) &- J(0,b) + J(a,b)
			\leqslant
			C_2 \, I(a,0,b,0),\\
			C_1 \, I(0,A,b,0)
			\leqslant
			J(A,0) - J(0,0) &- J(A,b) + J(0,b)
			\leqslant
			C_2 \, I(0,A,b,0).		
		\end{align*}
		Summing these four expressions, we rigorously obtain \eqref{eq:lem12}, which finishes the proof in the case $\alpha<0<A$ and $\beta<0<B$.
		The remaining cases can be covered by the same analysis, and it is even simpler since we have less subrectangles, so we omit further details.
	\end{proof}
	
	\begin{remark}
		One can prove by similar arguments that the integral representation \eqref{eq:four_point_integral_representation} holds for any $a<A$ and $b<B$.
	\end{remark}

	Let us now provide a simple fact which we use in the proof of Lemma~\ref{lem:Pa-Sobolev}.
	
	\begin{lemma}\label{lem:positive-part}
		Let $u \in \Wo$ be a nonnegative function.
		Then there exists a sequence $\{u_n\} \subset C_0^\infty(\Omega)$ of nonnegative functions converging to $u$ in $\Wo$.
	\end{lemma}
	\begin{proof}
		It follows from the definition of $\Wo$ that there exists $\{v_n\} \subset C_0^\infty(\Omega)$ converging to $u$ in $\Wo$.
		By \cite[Theorem~2.7]{BLP}, we have $v_n \to u$ in $L^p(\Omega)$, up to a subsequence.
		Let us consider the sequence of positive parts $\{v_n^+\}$. 
		It is not hard to see that each $v_n^+$ is a Lipschitz functions with compact support in $\Omega$, that is, $\{v_n^+\} \subset C_0^{0,1}(\Omega)$.
		Since $|a^+-b^+| \leqslant |a-b|$ for any $a,b \in \mathbb{R}$, we get $v_n^+ \to u^+ \equiv u$ in $L^p(\Omega)$ and $[v_n^+]_p \leqslant [v_n]_p$ for any $n$ (cf.\ \cite{musnaz} for elaboration).
		Consequently, $\{v_n^+\}$ is bounded in $\Wo$ 
		and hence converges weakly in $\Wo$ to a function $v \in \Wo$, up to a subsequence.
		We again deduce from \cite[Theorem~2.7]{BLP} that $v_n^+ \to v$ in $L^p(\Omega)$, up to a subsequence, which yields $v = u$. 
		If $v_n^+ \to u$ in $\Wo$, then, recalling that each $v_n^+$ has a compact support, we can approximate $v_n^+$ by nonnegative $C_0^\infty(\Omega)$-functions in the norm of $\W$ via mollification (see \cite[Lemma~11]{FSV}). 
		Taking a diagonal sequence, we obtain the desired claim.
		If $v_n^+ \to u$ only weakly in $\Wo$ (and not strongly), then we apply Mazur's lemma to obtain a sequence $\{w_n\}$ consisting of finite \textit{convex} combinations of $v_n^+$'s which converge to $u$ in $\Wo$. 
		In particular, any $w_n$ is nonnegative and belongs to $C_0^{0,1}(\Omega)$.
		Arguing as above, we finish the proof.
	\end{proof}
	
	\begin{remark}\label{rem:W-Lip}
		Since $\Wo$ is the completion of $C_0^\infty(\Omega)$ with respect to the norm $[\,\cdot\,]_p$ (see \cite[Remark~2.5]{BLP}), $\Wo$ can be equivalently defined as the completion of the space $C_0^{0,1}(\Omega)$ of Lipschitz functions with compact support in $\Omega$ with respect to $[\,\cdot\,]_p$. 
		Indeed, it is not hard to check that $[w]_p < \infty$ for any $w \in C_0^{0,1}(\Omega)$, and hence $w \in \W$. Then, recalling that $w$ has a compact support in $\Omega$, we apply mollification arguments (see \cite[Lemma~11]{FSV}) to conclude that $w$ can be approximated by $C_0^\infty(\Omega)$-functions in the norm $[\,\cdot\,]_p$.
	\end{remark}

	Finally, we provide two auxiliary lemmas needed to prove Propositions~\ref{prop:lambda2-characterization} and~\ref{prop:altern-char-lens}, respectively.
	\begin{lemma}[\cite{BrPar}]\label{lem:appendix:inequality}
		Let $U, V \in \mathbb{R}$ be such that $U V \leqslant 0$.
		Then 
		\begin{equation}\label{eq:proof:main-ineq0}
			|U-V|^{p-2} (U-V) (|\alpha|^p U - |\beta|^p V)
			\geqslant
			|\alpha U - \beta V|^p
		\end{equation}
		for any $(\alpha,\beta) \in S^1$.
		Moreover, equality holds in \eqref{eq:proof:main-ineq0} if and only if $U V =0$ or $\alpha = \beta$.
	\end{lemma}
	\begin{proof}
		The proof of the inequality \eqref{eq:proof:main-ineq0} is contained in the proof of \cite[Proposition~4.2]{BrPar} (see, more precisely, the proof of \cite[Eq.~(4.7), pp.~346-347]{BrPar}) by noting that \eqref{eq:proof:main-ineq0} coincides with \cite[Eq.~(4.7)]{BrPar} (via factoring out $\omega_1$ and $\omega_2$ in \cite[Eq.~(4.7)]{BrPar} by the homogeneity).
		Although unstated explicitly in \cite{BrPar}, 
		the second part of the lemma follows from \cite{BrPar}
		by inspection of the arguments.
		We omit details.
	\end{proof}

	\begin{lemma}\label{lem:appendix:inequality2}
		Let $U, V \in \mathbb{R}$ be such that $U V \leqslant 0$.
		Then 
		\begin{align}
			\label{eq:proof:main-ineq01}
			|U - V|^{p-2} (U- V) U
			&\geqslant
			|U-s V|^{p-2} (U- s V) U,\\  
			\label{eq:proof:main-ineq02}
			|U - V|^{p-2} (U- V) (-V)
			&\geqslant
			|sU-V|^{p-2} (sU- V) (-V),
		\end{align}
		for any $s \in [0,1]$.
	\end{lemma}
	\begin{proof}
		Define a continuous function $h:[0,1] \to \mathbb{R}$ as $h(s)=|U-s V|^{p-2} (U- s V) U$.
		We see that $h'(s) = -(p-1)|U-s V|^{p-2} UV \leqslant 0$ whenever $U-s V \neq 0$, i.e., $h$ is nondecreasing.
		Since $h(1) \geqslant h(0)$ by Lemma~\ref{lem:appendix:inequality} with $(\alpha,\beta)=(1,0)$, 
		we conclude that $h(1) \geqslant h(s)$ for all $s \in [0,1]$, which is exactly \eqref{eq:proof:main-ineq01}.
		In the same way, one can establish \eqref{eq:proof:main-ineq02}.
	\end{proof}

\end{document}